\documentclass[12pt,a4paper]{amsart}
\usepackage{amsfonts}
\usepackage{amsthm}
\usepackage{amsmath}
\usepackage{amscd}
\usepackage{t1enc}
\usepackage{indentfirst}
\usepackage{graphicx}
\usepackage{pict2e}
\usepackage{epic}
\numberwithin{equation}{section}
\usepackage[margin=2cm]{geometry}
\usepackage{epstopdf} 
\usepackage[utf8]{inputenc} 
\usepackage[T1]{fontenc}
\usepackage{array}
\usepackage{subcaption}

\usepackage{float}

\usepackage{enumitem}

\usepackage{comment} 

\usepackage{tikz}
\usetikzlibrary{graphs}
\usetikzlibrary{graphs.standard}

\usepackage[colorlinks = true,
linkcolor = blue,
urlcolor  = blue,
citecolor = blue,
anchorcolor = blue]{hyperref}

\numberwithin{equation}{section}

\usetikzlibrary{positioning}
\tikzset{main node/.style={circle,fill=white!20,draw,minimum size=1cm,inner sep=0pt},
	}	
\usetikzlibrary {graphs.standard}	

\theoremstyle{plain}
\newtheorem{Th}{Theorem}[section]
\newtheorem{Lemma}[Th]{Lemma}
\newtheorem{Cor}[Th]{Corollary}
\newtheorem{Prop}[Th]{Proposition}

\theoremstyle{definition}

\newtheorem{Def}[Th]{Definition}
\newtheorem{Que}{Question}

\newtheorem{Rem}[Th]{Remark}
\newtheorem{?}[Th]{Problem}
\newtheorem{Ex}[Th]{Example}
\newcommand{\rk}{\operatorname {rk}}
\newtheorem*{Nt*}{Note}

\begin{document}

	 	\setcounter{page}{1}
	 \vspace{2cm}
	 
	 \author{ Gökçen Dilaver$^{1*}$ , Selma Altınok$^2$ }
	 \title{ Extending Generalized Splines over the integers}

	 \thanks{\noindent $^1$ Hacettepe University Department of Mathematics, 06800 Beytepe, Ankara, Turkey\\
	 	\indent \,\,\, e-mail: gokcen.dilaver11@hacettepe.edu.tr ; ORCID: https://orcid.org/0000-0001-6055-111X. \\
	 		\indent  $^*$ \, Corresponding author. \\
	 	\indent $^2$ Hacettepe University Department of Mathematics, 06800 Beytepe, Ankara, Turkey\\
	 	\indent \,\,\,  e-mail: sbhupal@hacettepe.edu.tr ; ORCID: https://orcid.org/0000-0002-0782-1587. \\
	 	}

	 	\begin{abstract}
	 	Let $R$ be a commutative ring with identity and $G$ a graph. \emph{An extending generalized spline} on $G$ is a vertex labeling $f \in \prod_{v} M_v$ such that at each edge $e=uv$ there is an $R$-module $M_{uv}$ together with homomorphisms $ \varphi_u : M_u \to M_{uv}$ and  $ \varphi_v : M_v \to M_{uv}$  for each vertex $u, v$ incident to the edge $e$ so that $\varphi_u(f_u)=\varphi_v(f_v).$ Extending generalized splines are further generalizations for generalized splines. They can also be considered as generalized splines over modules.
	 
	 The main goal of this paper is to study the $R$-module structure of extending generalized splines. We concentrate on two following questions: which of the results for general splines extend to generalized splines over modules and if there is an algorithm or an explicit formula for special basis classes, called a flow up basis, for generalized splines over modules. We show that certain results concerning generalized splines can be extended to a  setting where each vertex $v$ is assigned a module $M_v=m_v\mathbb Z$. We provide an algorithm to construct a special basis for generalized splines over these modules on  paths.  Additionally, we introduce a new technique to construct a flow-up basis on arbitrary graphs using the idea of an algorithm on paths. 
	 \end{abstract}

	 \maketitle



\section{Introduction} 

Engineers first used the term "spline" to model complicated objects like cars and ships.  Mathematicians later adopted the term and they refer to such objects as "classical splines". Classical splines are piecewise polynomial functions that agree on the intersection of adjacent faces on the faces of a polyhedral complex. In topology and geometry, classical splines are important tools to construct equivariant cohomology rings(see Goresky, Kottwitz and MacPherson\cite{GKM}) and also recently in many areas like, computer-based animation and geometric design.

In this paper, we explore the theoretical framework of extended generalized splines. The concept of extending generalized splines was initially introduced in the open questions section of the work by Gilbert, Polster, and Tymoczko \cite{GPT}, and later elaborated on by Tymoczko \cite{Tymoczko}. Unlike classical generalized splines, these extended versions assign modules $M_v$ to the vertices instead of using the ring $R$ itself. More precisely, let $R$ be a commutative ring with unity, and let $G = (V, E)$ denote an arbitrary graph with vertex set $V$ and edge set $E$.
 \emph{An extending generalized spline} on $G$ is a vertex labeling $f \in \prod_{v\in V} M_v$ such that at each edge $e=uv\in E$ there is an $R$-module $M_{uv}$ together with homomorphisms $\varphi_u : M_u \to M_{uv}$ and  
$\varphi_v : M_v \to M_{uv}$ so that $\varphi_u(f_u)=\varphi_v(f_v).$
The set of extending generalized splines of $G$ is given by
$$\hat{R}_{G}= \{f \in \prod_{v\in V} M_v \mid \text{for each edge }\: e=uv, \: \varphi_u(f_u)=\varphi_v(f_v)\}.$$
The extending generalized splines $\hat{R}_G$ has an $R$-module structure.
When each $M_v$ is $R$ and each $M_{uv}$ is $R/\beta(uv)$ where $\beta$ is a function assigned an ideal of $R$ to each edge $uv$ so that $\varphi_u : M_u \to M_{uv}$ is the standard quotient map, the definition of extending generalized splines  corresponds to the definition of generalized splines $R_{(G,\beta)}$ over $R$ (see Gilbert, Polster and Tymoczko \cite{GPT}).

 One of the most important questions in the study of generalized splines is determining their size and identifying their minimal generators. Most research focuses on a specific choice of the ring $R$, the graph $G$, or the edge labeling function $\beta$ to develop the general theory of generalized splines on $G$ over $R$ (see \cite{AD,AS2019,AS2021,BHKR,Gjoni,GPT,HMR,Mahdavi}).
 
 Handschy, Melnick, and Reinders \cite{HMR} studied integer-valued generalized splines on cycle graphs. They introduced a special class of generalized splines, called flow-up classes, and demonstrated the existence of minimal flow-up classes on cycles over $\mathbb{Z}$. Additionally, they proved that these flow-up classes form a basis for the module of generalized splines. Later, Altınok and Sarıoğlan \cite{AS2019} provided an explicit formula for constructing a flow-up basis for generalized splines on arbitrary graphs over a principal ideal domain (PID) $R$, using certain special trails called zero trails. Their approach showed that selecting minimal zero trails is sufficient to determine a flow-up basis.
 In this paper, we propose a new technique, called the longest path technique,  to determine a flow-up basis which also relies on special trails. However, in our case, choosing maximal trails is sufficient.

Gilbert, Polster, and Tymoczko \cite{GPT} posed several open questions regarding the extension of generalized splines. One central question is which aspects of the theory of generalized splines can be extended when the vertex labels are modules rather than elements of a ring. Another fundamental question concerns the existence of flow-up classes in this extended setting, and the conditions under which a flow-up basis can be constructed. In particular, they ask whether an algorithm can be developed to construct flow-up classes, and under what conditions these classes form a basis.

In this paper, we address these questions for generalized splines defined over submodules $M_v = r_v \mathbb{Z}$ of $\mathbb{Z}$, assigned to each vertex $v$. We provide explicit formulas for constructing a flow-up basis on path graphs and, using the same technique, extend these results to  arbitrary graphs.
Furthermore, our methods naturally extend to the setting in which each vertex $v$ is assigned a module of the form $M_v = m_v R$, where $R$ is a principal ideal domain (PID).

Firstly, we focus on the following questions and we provide some examples.

\begin{Que}
	Which of the previous results on generalized splines extend to extending generalized splines (or generalized splines over modules)?
\end{Que}

\begin{Que}
	Identify extending generalized splines $\hat{R}_G$  as a module when $R$ is a particular ring, such as  integers, polynomial rings, ring of Laurent polynomials, domain etc. Determine whether a flow-up class exists or not (see Section \ref{Preliminaries} for the definition).

\end{Que}

\begin{Ex} 
	Let	$\mathbb{Q}^{(p)} = \{\frac{m}{p^s} \mid m,s \in \mathbb{Z}, s \ge 0\}.$ The $p$-primary component of $ \mathbb{Q}^{(p)} / \mathbb{Z}$ is called the Prüfer group, it is denoted by $\mathbb{Z}(p^{\infty})$. Since it is not a ring we cannot  define generalized spline over $\mathbb{Z}(p^{\infty})$. 
	Never the less, we can define extending generalized splines. Now we consider a graph $G$ whose vertices $v_i$ are labeled by $\mathbb{Z}$-modules $M_{v_i}= \bigg(\frac{1}{p^i}\mathbb{Z}\bigg) / \mathbb{Z}$ and edges $v_iv_j$ are labeled by modules $M_{v_iv_j} := \mathbb{Z}(p^{\infty})$ together with identity homomorphisms $M_{v_i}\to M_{v_iv_j}$ for each vertex $v_i$ incident to the edge ${v_iv_j}$. An extending generalized spline on $G$ is a vertex labeling $f \in \prod_{v_i} M_{v_i}$ satisfying  $$\varphi_{v_i}(f_{v_i})=\varphi_{v_j}(f_{v_j}).$$ It is easy to see that  $\hat{R}_G$ only contains constant splines. 
	 	
	\begin{figure}[h!]
		
		\begin{center}
			
			\begin{tikzpicture}
				\node[main node] (1) {$\bigg(\frac{1}{p^i}\mathbb{Z}\bigg) / \mathbb{Z}$};
				\node[main node] (2) [right = 3cm of 1]  {$\bigg(\frac{1}{p^j}\mathbb{Z}\bigg) / \mathbb{Z}$};
				
				\path[draw,thick]
				(1) edge node [above]{$\mathbb{Z}(p^{\infty})$} (2);

			\end{tikzpicture}
			
		\end{center}
		
		\caption{An edge-labeled path graph $P_2$} \label{prüfer}
	\end{figure}
	In particular, we consider the path graph $P_2$ as in Figure \ref{prüfer} together with $\bigg(\frac{1}{p^i}\mathbb{Z}\bigg) / \mathbb{Z}$ at each vertex $v_i$ and $\mathbb{Z}(p^{\infty})$ at the edge. Let $F=(f_1,f_2) \in \hat{R}_G$ be an extending generalized spline. By a spline condition,  we have
	$f_1= \overline{\frac{a}{p^i}} $ and $f_2= \overline{\frac{b}{p^j}} $ for some $a,b \in  \mathbb{Z}$ such that $f_1= f_2$. This implies that $$\frac{a}{p^i} - \frac{b}{p^j} \in \mathbb{Z}.$$
	Assume that $j=i+k$ for some $k$. Then 
	$(a-tp^i)p^k = b$ for some $t\in \mathbb{Z}$. Therefore, we obtain 
	$$F=(f_1,f_2) = (\overline{\frac{a}{p^i}}, \overline{\frac{ap^k}{p^j}}) = (\overline{\frac{a}{p^i}}, \overline{\frac{a}{p^i}}) $$ since $  \bigg(\frac{1}{p^i}\mathbb{Z}\bigg) / \mathbb{Z} \subset \bigg(\frac{1}{p^j}\mathbb{Z}\bigg) / \mathbb{Z}  \subset  \mathbb{Z}(p^{\infty})$.
	Thus, there exist only constant splines on graph $G$. Therefore, there is no  a second flow-up class.

\end{Ex}

	\begin{figure}[h!]
	
	\begin{center}
		
		\begin{tikzpicture}
			\node[main node] (1) {$2 \mathbb{Z}$};
			\node[main node] (2) [right = 3cm of 1]  {$3 \mathbb{Z}$};
			
			\path[draw,thick]
			(1) edge node [above]{$ \mathbb{Z}$} (2);

		\end{tikzpicture}
		
	\end{center}
	
	\caption{An edge-labeled path graph $P_2$} \label{extending}
\end{figure}	

\begin{Ex}
	Let $P_2$ be a path graph with modules $M_{v_1}= 2\mathbb{Z} $ and $M_{v_2}= 3\mathbb{Z}$ at vertices $v_1,v_2$ and $M_{v_1v_2}=\mathbb Z$ at the edge as in Figure \ref{extending}. An extending generalized spline on $P_2$ is a vertex labeling together with identity $\mathbb{Z}$-module homomorphisms
	\begin{center}

		\begin{tabular}{c c c}
			$ \varphi_{v_1} : 2 \mathbb{Z} \to \mathbb{Z} $ & $\quad $& $ \varphi_{v_2} :3 \mathbb{Z} \to \mathbb{Z}$ \\
			$\quad \quad 2k \to 2k$ & & $\quad \quad 3k \to 3k$
		\end{tabular}
	\end{center}
	suct that   $\varphi_{v_1}(f_{v_1})=\varphi_{v_2}(f_{v_2}).$ Then, we easily observe that  $\hat{R}_G$ is a free $\mathbb{Z}$-module $<(6,6) >$ with rank 1. Therefore, there is no a second flow-up class.

\end{Ex}

	\begin{figure}[h!]
	
	\begin{center}
		
		\begin{tikzpicture}
			\node[main node] (1) {$M$};
			\node[main node] (2) [right = 3cm of 1]  {$M$};
			\node[main node] (3) [right = 3cm of 2]  {$M$};
			
			\path[draw,thick]
			(1) edge node [above]{$M$} (2)
			(2) edge node [above]{$M$} (3);

		\end{tikzpicture}
		
	\end{center}
	
	\caption{An edge-labeled path graph $P_3$} \label{matrix}
\end{figure}
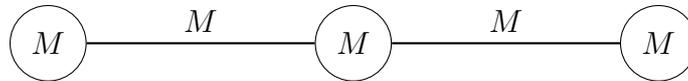

\begin{Ex}
Let $R= \mathbb{Z}_2  \times \mathbb{Z}_2$ and $S=M_2(R)$. Let
$$e_{11} =\begin{bmatrix}
	(1,1)       & \overline{0}  \\
\overline	0       & \overline{0} 

\end{bmatrix},$$
 where $\overline{0}=(0,0)$ and $M=  e_{11} S$. Observe that $e_{11}$ is an idempotent element and $M$ is a right $S$-module. We consider an edge-labeled graph  $G$ whose vertices $v_i$ are labeled by an $S$-module $M_{v_i}= M$ and edges $v_iv_j$ are labeled by the module $M$ together with $S$-module  homomorphisms $\varphi_{v_iv_j}: M\to M$:
\begin{center}
	\begin{tabular}{c c c c}
		$ \varphi_{v_1v_2} : M \to M, \quad $  &   $\varphi_{v_2v_1} :M \to M ,   \quad \quad $ &  $ \varphi_{v_2v_3} :M \to M , \quad $ & $ \varphi_{v_3v_2} :M \to M$  \\
		
		$\quad \quad f_{v_1} \to e_{11}f_{v_1} $ &  $f_{v_2} \to e_{10}f_{v_2}$ &  $\quad \quad f_{v_2} \to e_{01}f_{v_2}$ & $\quad \quad f_{v_3} \to e_{10}f_{v_3}$ 
	\end{tabular}
\end{center}
  where
$$ \varphi_{v_1v_2}(f_{v_1})= e_{11}f_{v_1}= \begin{bmatrix}
	(1,1)       & \overline{0}  \\
 \overline{0}     &  \overline{0}
	
\end{bmatrix} \begin{bmatrix}
 (a_1,a_2)     &  (b_1,b_2) \\
	 \overline{0}       &  \overline{0}
	
\end{bmatrix}  = \begin{bmatrix}
	(a_1,a_2)      & (b_1,b_2)  \\
 \overline{0}       &  \overline{0}
	
\end{bmatrix} ,$$
$$  \varphi_{v_2v_1}(f_{v_2})= e_{10}f_{v_2}= \begin{bmatrix}
	(1,0)       &  \overline{0}  \\
 \overline{0}       &  \overline{0} 
	
\end{bmatrix} \begin{bmatrix}
  (c_1,c_2)      & (d_1,d_2)  \\
	 \overline{0}       &  \overline{0} 
	
\end{bmatrix} = \begin{bmatrix}
	(c_1,0)      & (d_1,0)  \\
	 \overline{0}       &  \overline{0}
	
\end{bmatrix},$$

$$  \varphi_{v_2v_3}(f_{v_2})= e_{01}f_{v_2}= \begin{bmatrix}
	(0,1)       &  \overline{0}  \\
 \overline{0}      &  \overline{0} 
	
\end{bmatrix} \begin{bmatrix}
	(c_1,c_2)     & (d_1,d_2)  \\
 \overline{0}       &  \overline{0}
	
\end{bmatrix} = \begin{bmatrix}
	(0,c_2)      & (0,d_2)  \\
 \overline{0}       &  \overline{0}
	
\end{bmatrix},$$

$$  \varphi_{v_3v_2}(f_{v_3})= e_{10}f_{v_3}= \begin{bmatrix}
	(1,0)       & 0  \\
	\overline{0}      & \overline{0}  
	
\end{bmatrix} \begin{bmatrix}
	(k_1,k_2)      & (l_1,l_2)  \\
	\overline{0}   & \overline{0}
	
\end{bmatrix} = \begin{bmatrix}
	(k_1,0)      & (l_1,0)  \\
	 \overline{0}      & \overline{0}
	
\end{bmatrix}.$$
We can determine an extending generalized spline $(g_1,g_2,g_3) \in \hat{R}_{P_3}$ satisfying  \begin{align*}
\varphi_{v_1v_2}(g_1) &=\varphi_{v_2v_1}(g_2), \\ \varphi_{v_2v_3}(g_2) &=\varphi_{v_3v_2}(g_3).
\end{align*}
Under these conditions, we obtain $$ g_1=g_2 =  \begin{bmatrix}
	(a_1,0)     & (b_1,0)  \\
 \overline{0}     &  \overline{0} 
	
\end{bmatrix},  \quad \quad 
 g_3 =  \begin{bmatrix}
 	(0,k_2)     & (0,l_2)  \\
 	 \overline{0}     &  \overline{0}  
 	
 \end{bmatrix}.$$
 As in the example above, we cannot find a second flow-up class here either.
\end{Ex}

\begin{figure}[h!]
	
	\begin{center}
		
		\begin{tikzpicture}
			\node[main node] (1) {$3 \mathbb{Z}$};
			\node[main node] (2) [right = 3cm of 1]  {$2 \mathbb{Z}$};
			
			\path[draw,thick]
			(1) edge node [above]{$\mathbb{Z}/7 \mathbb{Z}$} (2);

		\end{tikzpicture}
		
	\end{center}
	
	\caption{An edge-labeled path graph $(P_2,\beta)$} \label{smallest}
\end{figure}
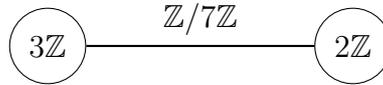

\begin{Ex}
	A definition of the smallest flow-up class in \cite{HMR} does not work here. For example,  consider the edge-labeled path graph $(P_2,\beta)$ as in Figure 
	\ref{smallest}. By Corollary \ref{$g_i with n$}, $(3,10)$ is a minimal spline but not the smallest flow-up class with respect to their definition since  $(6,6)$ is another spline for this path graph.
\end{Ex}

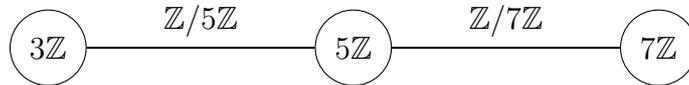
\begin{figure}[h!]
	
	\begin{center}

		\begin{tikzpicture}
			\node[main node] (1) {$3 \mathbb{Z}$};
			\node[main node] (2) [right = 3cm of 1]  {$5 \mathbb{Z}$};
			\node[main node] (3) [right = 3cm of 2] {$7 \mathbb{Z}$};
			
			\path[draw,thick]
			(1) edge node [above]{$\mathbb{Z}/5 \mathbb{Z}$} (2)
			(2) edge node [above]{$\mathbb{Z}/7 \mathbb{Z}$} (3);

		\end{tikzpicture}
		
	\end{center}
	
	\caption{An edge-labeled path graph $(P_3,\beta)$}
	\label{trail}
\end{figure}

\begin{Ex}
In \cite{AS2019}, Altınok and Sarıoğlan introduce the concept of zero trails to identify minimal flow-up classes on arbitrary graphs over a principal ideal domain. However, their algorithm for finding a flow-up class with the smallest nonzero entry does not apply in this context. For instance, when applying zero trails to compute a minimal flow-up class $F^{(2)} = (0, f_{v_2}^{(2)}, f_{v_3}^{(2)})$ on the edge-labeled path graph $P_3$ (as shown in Figure \ref{trail}), we obtain $f_{v_2}^{(2)} = 5$ (see Theorem 3.8 in \cite{AS2019}). Yet, no value of $f_{v_3}^{(2)}$ satisfies both spline conditions: $5 \equiv f_{v_3}^{(2)} \mod 7$ and $f_{v_3}^{(2)} \in 7\mathbb{Z}$. Hence, such an $F^{(2)}$ cannot exist—contradicting the known existence of a minimal flow-up class $F^{(2)} = (0, 35, 0)$.

\end{Ex}


	
	\section{Background and Notations}\label{Preliminaries}

We use $P_n$ to denote a path graph with $n$ vertices; $C_n$ to denote a cycle graph with $n$ vertices; $T_n$ to denote a tree graph with $n$ vertices and $S_n$  to denote a star graph with $n$ vertices. We use the notation $( \quad )$
for the greatest common divisor and $[\quad ]$ for the least common multiple.
\begin{Def}
	
	Let $R$ be a ring and $G = (V,E)$ a graph. An edge-labeling function is a map $ \beta : E \to \{\text{R-modules}\}$ which assigns an $R$-module $M_{uv}$  to each edge $uv$. We call the pair $(G, \beta)$ as an \emph{edge-labeled graph}.
	
\end{Def}

\begin{Def}

	Let $R$ be a ring and $(G,\beta)$ an edge-labeled graph. \emph{An extending generalized spline} on $(G,\beta)$ is a vertex labeling $f \in \prod_{v} M_v$ such that for each edge $uv$ there is an $R$-module $M_{uv}$  together with $R$-module homomorphisms  $ \varphi_u : M_u \to M_{uv}$ and  $ \varphi_v : M_v \to M_{uv}$ satisfying  $\varphi_u(f_u)=\varphi_v(f_v).$ The collection of extending generalized splines on $(G,\beta)$ with the base ring $R$ is denoted by ${\hat R}_{G}$:
	$${\hat R}_{G}= \{f \in \prod_{v} M_v \mid \text{for each edge}\: uv, \: \varphi_u(f_u)=\varphi_v(f_v)\}.$$
	For elements of $\hat{R}_G $ we use a column matrix notation with ordering from bottom to top as follows: 
\begin{equation*}
		F =
	\left( \begin{array}{c}
		f_{v_n} \\
		\vdots\\
		f_{v_1}
	\end{array} \right) \in \hat{R}_G.
\end{equation*}
 We also use a vector notation as $	F= (f_{v_1},\ldots,f_{v_n}).$	
	For the simplicity,  we refer to extending generalized splines as splines. 
\end{Def}

\begin{Rem}
	
	$\hat{R}_G$ has an $R$-module structure by the following two binary operations:	
	\begin{align*}	
			\hat R_G \times \hat{R}_G \to \hat{R}_G ,
			& \, (f,g) \mapsto f+g:= (f_{v_1}+g_{v_1},f_{v_2}+g_{v_2},\ldots,f_{v_n}+g_{v_n}),\\
			R \times \hat{R}_G \to  \hat{R}_G, & \, (r,f) \mapsto rf:= (rf_{v_1},rf_{v_2}, \ldots, rf_{v_n})
\end{align*}
where $f=(f_{v_1},\ldots,f_{v_n})$ and $g=(g_{v_1},\ldots,g_{v_n})$.

\end{Rem}

\begin{Def}
	A  spline  $F= (f_{v_1},\ldots,f_{v_n}) \in  \prod_{v_i} M_{v_i} $ is\emph{ non-trivial} if  $f_{v_j}$ is nonzero for at least one of the $j$. Note that $(0,\ldots,0)$ is a trivial spline. 
\end{Def}
	In this paper, we consider a graph $G$ whose vertices are labeled by $\mathbb{Z}$-modules $M_v= m_v \mathbb{Z}$ and edges $e$ labeled by $\mathbb{Z}$-modules $M_e := \mathbb{Z}/r_e\mathbb Z$ together with quotient homomorphisms $ \varphi_v : M_v \to M_e$, $a \to a+r_e\mathbb Z$, for each vertex $v$ incident to $e$ unless otherwise stated. We study extending generalized splines ${\hat R}_G$ on $G$ satisfying $\varphi_u(f_u)=\varphi_v(f_v).$

 For simplicity, we associate $r_{e} \mathbb{Z}$ to each edge label $ l_{e}=\beta(e)= \mathbb{Z} / r_{e} \mathbb{Z} $.

\begin{Lemma}[Existence of non-trivial splines]
	Let $(G,\beta)$ be an edge-labeled graph with a $\mathbb{Z}$-module  $M_v= m_v \mathbb{Z}$  at each vertex where  $\beta(uv)= \mathbb{Z}/r_{uv} \mathbb{Z} \subset \mathbb{Z}$ at each edge $uv$. Then there exists a nontrivial spline on $(G,\beta)$.
\end{Lemma}

\begin{proof}
	
	Let $N_{i}= \{r_{v_iv_j} : v_iv_j \in E \}$.
	Assume that the $f_{v_j}$ are labeled by zeros for all $j \neq 1$. If $v_1$ is connected with $v_j$ then we have 
	\begin{gather*}
		f _{v_1} \equiv 0 \mod r_{v_1v_j} \Rightarrow f_{v_1} \in r_{v_1v_j} \mathbb{Z}.
	\end{gather*}
	By definition, we know that $f_{v_1} \in m_{v_1} \mathbb{Z}$.
	Therefore,  $F=(f_{v_1},0,\ldots,0) \in \prod_{v} M_{v}$  with $f_{v_1} = [ m_{v_1}, N_{1}]$ gives a non-trivial  spline on $(G,\beta)$.
	
\end{proof}

\begin{Def}
	Let  $(G,\beta)$  be an edge-labeled graph with an $R$-module $M_v$ at each vertex $v$ and $\beta (uv)= M_{uv}$  at each edge $uv$. \emph{An $i$-th flow-up class} $F^{(i)}$ over $(G,\beta)$ with $1\le i \le n $ 
	is a spline in $\hat{R}_G$ for which the component $f_{v_i}^{(i)} \neq 0 $ and $f_{v_s}^{(i)} = 0 $ whenever $s < i $. The set of all $i$-th  flow-up classes is denoted by $\hat{\mathcal{F}_i}$.
\end{Def}

\begin{Def}
	\emph{A constant flow-up class}  is a flow-up class  $F^{(i)} \in  \prod_{v_i} M_{v_i} $ with $f^{(i)}_{v_s}=a \in M_{v_s}$ for each vertex $v_s \in V$ and $s \ge i$ .

\end{Def}

\begin{Lemma}[Existence of constant flow-up classes]
	Let $(G,\beta)$ be an edge–labeled graph with  a $\mathbb{Z}$-module $M_v= m_v \mathbb{Z}$  at each vertex where  $\beta (uv)= \mathbb{Z}/r_{uv} \mathbb{Z}$ at each edge $uv$. Then there exists a constant flow-up class on $(G,\beta)$.
\end{Lemma}

\begin{proof}
	
	Let $a = [N_{i},M_{i}] $ where  $N_{i}= \{r_{v_kv_j} : v_kv_j \in E \: \text{ for } \: j<i \: \text{ and } \: k \ge i\}$ and $M_{i} =\{m_{v_k} : k \ge i\}$.
	Therefore,  $F^{(i)}= (0,\ldots,0,f_{v_i}^{(i)},\ldots,f_{v_n}^{(i)}) \in  \prod_{v_i} M_{v_i} $  with $f_{v_t}^{(i)} = a \in \mathbb{Z}$  for all $t=i,\ldots,n$ is a constant flow-up class.

\end{proof}

\begin{Def}
	Let $F^{(i)}=(0,\dots,0,f_{v_i}^{(i)}, \dots, f_{v_n}^{(i)}) \in \hat{\mathcal{F}_i}$   be a flow-up class with $f_{v_i}^{(i)} \neq 0$ for $i=1,2,\ldots,n$ and $f_{v_s}^{(i)} =0$ for all $s <i$. We define the leading term of $F^{(i)}, LT(F^{(i)})=f_{v_i}^{(i)}$, the first non-zero entry of $F^{(i)}$. 
	The set of all leading terms of splines in $\hat{\mathcal{F}_i}, LT(\hat{\mathcal{F}_i})$, is an ideal of $\mathbb{Z}$.
\end{Def}

\begin{Def}
	A flow-up class $F^{(i)} $ is called a minimal element of $ \hat{\mathcal{F}_i}$ if $LT(F^{(i)})$ is a generator of the ideal $LT(\hat{\mathcal{F}_i})$.
\end{Def}

\begin{Def}
	\emph{A minimum generating set} for a $\mathbb{Z}$-module $\hat{R}_G$ is a spanning set of splines with the smallest possible number of elements. The size of a minimum generating set is called  \emph{rank} and denoted by $\rk \hat{R}_G$.
\end{Def}

Bowden, Hagen, King and Reinders \cite{BHKR} gave the following formulation for $x$ in terms of $y, a$ and $b$ provided a solution does exist.

\begin{Prop}\label{Bowden}
	
	Consider the system of congruences 
	\begin{align*}
		x & \equiv y \mod a \\
		x & \equiv 0 \mod b.
	\end{align*} 
	If this system has a solution, then one solution is given by the following formula:
	
	\begin{itemize}
		\item If $\frac{a}{(a,b)}=1$ then $x=b$ is a solution to the system.
		
		\item If $\frac{a}{(a,b)} \neq 1$ then $$x = y \bigg(\frac{b}{(a,b)}\bigg) \bigg(\frac{b}{(a,b)}\bigg)^{-1} \mod \bigg(\frac{a}{(a,b)}\bigg)$$
		is a solution to the system.
		
	\end{itemize}
	
\end{Prop}

\begin{proof}
	
	See Proposition 2.9 in \cite{BHKR}.
\end{proof}

\begin{Th}[The Chinese Remainder Theorem] \label{CRT}
	Let $R$ be a PID and $x, a_1,\ldots, a_n, b_1, \ldots, b_n \in R$. Then the system
\begin{align*}
	x & \equiv a_1 \mod b_1 \\
	x & \equiv a_2 \mod b_2 \\
	 & \quad \quad \vdots\\
	x & \equiv a_n \mod b_n
\end{align*}
	has a solution if and only if $a_i \equiv a_j \mod (b_i, b_j )$ for all $i, j \in \{1,\dots,n\}$ with $i \neq j$.
	The solution is unique modulo $[b_1,\dots,b_n]$.

\end{Th}

\section{The Matrix $M_{G}$} \label{matrixsection}
In \cite{GPT}, Gilbert, Polster and Tymoczko gave a definition of the GKM matrix to describe all generalized splines. Under our assumption, we can reformulate and extend the GKM construction to extending generalized splines. Since $\hat R_G$ is a module this section allows to apply the theory of Groebner bases to the problem of calculating and analyzing splines over a ring of integers or polynomials ring. This computation could be done using Maple Groebner package, Mathematica or Macaulay 2.

Let $R$ be a commutative ring with identity and $G$ a graph with $n$ vertices and $s$ edges. We assume that $(G,\beta)$  is an edge-labeled graph with a cyclic  $R$-module $M_v= m_{v}R$ at each vertex $v$ and $\beta (uv)=R/ r_{uv}R$ at each edge $uv$ together with quotient homomorphisms $M_v \to M_{uv}$ at each vertex incident to the edge $uv$. Our definition of the GKM matrix is the following:

\begin{Def}
	Let  $(G,\beta)$  be an edge-labeled graph	
	and $	F= (f_{v_1},\ldots,f_{v_n}) \in \hat{R}_G$. Then for each
	$v_iv_j \in E(G)$ we have $f_{v_i}-  f_{v_j}= g_{ij}r_{ij}$ for some  $ g_{ij} \in \mathbb{Z}$ and by definition for each $v_i \in V(G)$ we have $f_{v_i} = m_{v_i}k_i$ for some $k_i \in \mathbb{Z}$. In other words, we have the spline conditions
	\begin{align*}
		f_{v_i}-  f_{v_j}+ g_{ij}r_{ij} & = 0 \\
		f_{v_i} +m_{v_i}k_i & = 0
	\end{align*}
 for some $ g_{ij},k_i \in \mathbb{Z}$. From here we define the following matrices.
	Let  $A$ be an $s \times n$ matrix with entries
	$$a_{ij}= \left\{
	\begin{array}{c}
		1, \quad \text{if} \quad j=x \\
		-1, \quad \text{if} \quad  j=y\\
		\quad 0, \quad \text{otherwise}
	\end{array}
	\right.$$
where the row corresponds to each edge $e=v_xv_y$ with $x< y$;
$B$ an $s\times s$ diagonal matrix with each nonzero entry $r_{v_xv_y}$ corresponding to edge $e=_{}v_xv_y$ 
and	$C$ an $n\times n$ diagonal matrix $(m_{v_1},\ldots,m_{v_n})$.
Then an extending GKM matrix $M_{(G,\beta)}$ of the pair $(G,\beta)$ is an $(n+s) \times (2n+s)$ matrix defined as
	$$	M_{(G,\beta)} = \begin{bmatrix}
		\begin{bmatrix}
			a_{ij}
		\end{bmatrix}_{s \times n} & \begin{bmatrix}
			b_{ij}
		\end{bmatrix}_{s \times s} & \begin{bmatrix}
			0
		\end{bmatrix}_{s \times n}\\
		\\
		\begin{bmatrix}
			I_{ij}
		\end{bmatrix}_{n \times n} &  \begin{bmatrix}
			0
		\end{bmatrix}_{n \times s} &  \begin{bmatrix}
			c_{ij}
		\end{bmatrix} _{n \times n}
	\end{bmatrix}
	$$ 
	where 
	$	\begin{bmatrix}
		I_{ij}
	\end{bmatrix}_{n \times n}$ is an $n\times n$ identity matrix. As long as there is no confusion, an extending GKM matrix $M_{(G,\beta)}$ is referred to as the GKM matrix $M_{G}$.

\end{Def}
The following proposition relates splines to the GKM matrix. 

\begin{Prop}\label{syz}
	Let $(G,\beta)$  be an edge-labeled graph. Then $(f_{v_1},\dots,f_{v_n}) \in \hat{R}_G$ if and only if there exists an $n+s$ tuple $(f_{v_{n+1}},\dots,f_{v_{2n+s}}) \in R^{n+s}$ such that $$(f_{v_1},\dots,f_{v_n},f_{v_{n+1}},\dots,f_{v_{2n+s}}) \in \ker M_G$$ where $M_G: R^{2n+s}\to R^{n+s}$ is a map corresponding to $M_G$.	
\end{Prop}

\begin{proof}
	Let $F=(f_{v_1},\dots,f_{v_n}) \in \hat{R}_G$. Then we have $f_{v_i} - f_{v_j} + f_{v_{n+t}}r_t =0$
	for some $f_{v_{n+t}} \in R$ at each edge $e_{ij}$ for $t \in \{1,\dots,s\}$ and 
	$f_{v_i} + f_{v_{n+s+j}}m_{v_i} = 0 $
	for some 	$f_{v_{n+s+j}} \in R$ at each vertex $v_i$ for $j  \in \{1,\dots,n\}$. We denote $H=(f_{v_1},\dots,f_{v_n},f_{v_{n+1}},\dots,f_{v_{2n+s}})$. Then it can be easily seen that $H \in \ker M_G$. Conversely, we assume that $H=(f_{v_1},\dots,f_{v_n},f_{v_{n+1}},\dots,f_{v_{2n+s}}) \in \ker M_G$.
	Then the product $M_GH^T$ is an $(n+s)\times 1$ zero vector.  It gives $n+s$ equations which correspond to the spline conditions. Thus, $F\in \hat R_G$.
	
\end{proof}

\begin{Prop}
	Let $(G,\beta)$  be an edge-labeled graph. Let $\{\alpha_i \mid i =1,\dots, 2n+s\}$ be the set of the columns of $M_G$. Then 
	$$\ker M_G = Syz(\alpha_1,\dots,\alpha_{2n+s}).$$
	Moreover, $\hat{R}_G = \pi_n(Syz(\alpha_1,\dots,\alpha_{2n+s}))$ where $\pi_n : R^{2n+s} \to R^n$ is the projection map.
\end{Prop}

\begin{proof}
	For the GKM matrix $M_G$ as defined above, we have
	\begin{gather*}
		(f_{v_1},\dots,f_{v_n},f_{v_{n+1}},\dots,f_{v_{2n+s}}) \in \ker M_G \Leftrightarrow M_G(f_{v_1},\dots,f_{v_n},f_{v_{n+1}},\dots,f_{v_{2n+s}})^T =0 \\
		\Leftrightarrow f_{v_1}\alpha_1 + \dots + f_{v_n} \alpha_n+ f_{v_{n+1}} \alpha_{n+1} + \dots+ f_{v_{n+s}}\alpha_{n+s}+ \dots + f_{v_{2n+s}} \alpha_{2n+s}=0 \\
		\Leftrightarrow 	(f_{v_1},\dots,f_{v_n},f_{v_{n+1}},\dots,f_{v_{2n+s}})
		\in	Syz(\alpha_1,\dots,\alpha_{2n+s}).
	\end{gather*}
	Moreover, using $\ker M_G = Syz(\alpha_1,\dots,\alpha_{2n+s})$ and Proposition \ref{syz}, it can be easily seen that $$\hat{R}_G = \pi_n(Syz(\alpha_1,\dots,\alpha_{2n+s}))$$ where $\pi_n : R^{2n+s} \to R^n$ is a projection map.	
	
\end{proof}

These results show that the projection of the syzygy module specified on the columns of $M_G$ can be thought of as the spline module $\hat{R}_G$. This is another way of characterizing of $\hat{R}_G$.

	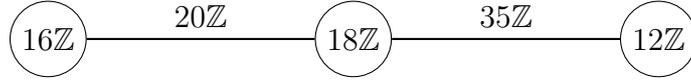
\begin{figure}[h]
		\begin{center}

			\begin{tikzpicture}
				\node[main node] (1) {$16 \mathbb{Z}$};
				\node[main node] (2) [right = 3cm of 1]  {$18 \mathbb{Z}$};
				\node[main node] (3) [right = 3cm of 2] {$12 \mathbb{Z}$};
				
				\path[draw,thick]
				(1) edge node [above]{$20 \mathbb{Z}$} (2)
				(2) edge node [above]{$35 \mathbb{Z}$} (3);

			\end{tikzpicture}

		\end{center}
		
		\caption{Example of an edge-labeled path graph $(P_3,\alpha)$} \label{matrixm}
	\end{figure}

	\begin{Ex}
		Let $(P_3,\beta)$ be an edge-labeled path graph as in Figure \ref{matrixm}. Then the GKM matrix $M_G$ is	
		$$M_G=	\begin{bmatrix}
			1 & -1 & 0 & 20 & 0 & 0 &0 &0  \\
			0 & 1 & -1& 0  & 35 & 0 & 0 & 0\\
			1 & 0 &  0 & 0  & 0 & 16 & 0 & 0\\
			0 & 1 & 0 & 0  & 0 & 0 & 18 & 0\\
			0 & 0 & 1& 0  & 0 & 0 & 0 & 12\\
		\end{bmatrix}.
		$$ 
		If we calculate $Syz(\alpha_1,\dots,\alpha_8)$ using Macaulay 2 and then take the projection of its, we find that 
		\begin{equation*}
			\begin{Bmatrix}
				\left( \begin{array}{c}
					6480\\
					180\\
					160
				\end{array} \right), &
				\left( \begin{array}{c}
					5184 \\
					144\\
					144
				\end{array} \right),
				&
				\left( \begin{array}{c}
					420\\
					0\\
					0
				\end{array} \right)
			\end{Bmatrix}
		\end{equation*}
		is a Groebner basis for $\hat{R}_{G}.$	
	\end{Ex}

\begin{figure}[h]
	\centering
	\begin{tikzpicture}
		\def \n {14}
		\def \N {6}
		\def \radius {4cm}
		\def \rd {3mm}
		\def \rer {8mm}
		
		\def \margin {8} 
		
		\node[draw, circle](1) at (360:0mm) {$15 \mathbb{Z}$};
		\node[draw, circle] (2) at ({360/\n *\n / 4}:\radius) {$7 \mathbb{Z}$};
		\node[draw, circle] (3)at ({360/4 - 360/\n * (3 - 1)}:\radius) {$15 \mathbb{Z}$};
		\node[draw, circle](4) at ({360/4 - 360/\n * (5 - 1)}:\radius) {$18 \mathbb{Z}$};
		\node[draw, circle] (5) at ({360/4 - 360/\n * (8 - 1)}:\radius) {$12 \mathbb{Z}$};

		\node[draw, circle] (6) at ({360/4 + 360/\n * (5 - 1)}:\radius) {$21 \mathbb{Z}$};

		\node[draw, circle](7) at ({360/4 + 360/\n * (3 - 1)}:\radius) {$9 \mathbb{Z}$};        
		
		\path[draw,thick]
		(1) edge node [right]{$  3 \mathbb{Z} $} (2)
		(1) edge node [right]{$ 8 \mathbb{Z}$} (3)
		(1) edge node [below]{$ 16 \mathbb{Z}$} (4)
		(1) edge node [right]{$ 15 \mathbb{Z}  $} (5)
		(1) edge node [below]{$24 \mathbb{Z}$} (6)
		(1) edge node [left]{$ 12 \mathbb{Z}$} (7);

	\end{tikzpicture}
	\caption{An edge-labeled star graph $(S_7,\beta)$} \label{starsyz}
	
\end{figure}
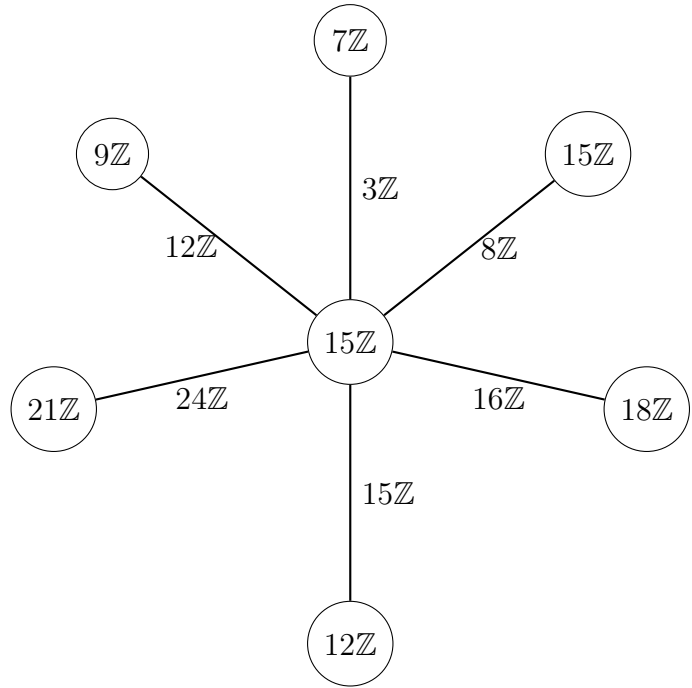

\begin{Ex}
	Let $(S_7,\beta)$ be an edge-labeled star graph as in Figure \ref{starsyz}. Then the GKM  matrix $M_G$ is
	
	$$	M_G = \begin{bmatrix}
		\begin{bmatrix}
			1 & -1& 0& \cdots & 0 \\
			1& 0 & -1 & \cdots & 0 \\
			\hdotsfor{5}\\
			1& 0 & \cdots &  0 & -1
		\end{bmatrix}_{6 \times 7} & \begin{bmatrix}
			3 & 0 & 0 & 0  & 0 & 0 \\
			0 & 8 & 0 & 0  & 0 & 0 \\
			0 & 0 & 16 & 0  & 0 & 0 \\
			0 & 0 & 0 & 15  & 0 & 0 \\	
			0 & 0 & 0 & 0  & 24 & 0 \\
			0 & 0 & 0 & 0  & 0 & 12 \\
		\end{bmatrix} & \begin{bmatrix}
			0
		\end{bmatrix}\\
		\\
		\begin{bmatrix}
			1 & 0 & \cdots & 0 & 0 \\
			0 & 1 &0 & \cdots & 0 \\
			\hdotsfor{5}\\
			0  &0 & \cdots & 0& 1 \\
			
		\end{bmatrix}_{7 \times 7}&  \begin{bmatrix}
			0
		\end{bmatrix} &  \begin{bmatrix}
			15 & 0 & 0 & 0  & 0 & 0 & 0 \\
			0 & 7 & 0 & 0  & 0 & 0 & 0\\
			0 & 0 & 15 & 0  & 0 & 0 & 0\\
			0 & 0 & 0 & 18  & 0 & 0 & 0\\	
			0 & 0 & 0 & 0  & 12 & 0 & 0\\
			0 & 0 & 0 & 0  & 0 & 21 & 0 \\
			0 & 0 & 0 & 0  & 0 & 0 & 9 \\
		\end{bmatrix} 
	\end{bmatrix}.
	$$ 	
	If we calculate $Syz(\alpha_1,\dots,\alpha_{20})$ using Macaulay 2 and then take the projection of its, we find that 
	\begin{equation*}
		\begin{Bmatrix}
			\left( \begin{array}{c}
				-630\\
				-1470\\
				-840\\
				1890\\
				-3150\\
				210\\
				210
				
			\end{array} \right), &
			\left( \begin{array}{c}
				-810\\
				-1890\\
				-1080\\
				2430\\
				-4050\\
				273\\
				270
			\end{array} \right),
			
			&
			\left( \begin{array}{c}
				0\\
				0\\
				0\\
				0\\
				120\\
				0\\
				0
			\end{array} \right),
			&
			\left( \begin{array}{c}
				0\\
				0\\
				0\\
				144\\
				0\\
				0\\
				0
			\end{array} \right),
			&

			\left( \begin{array}{c}
				0\\
				0\\
				60\\
				0\\
				0\\
				0\\
				0
			\end{array} \right),
			
			&
			\left( \begin{array}{c}
				0\\
				168\\
				0\\
				0\\
				0\\
				0\\
				0
			\end{array} \right),
			&
			\left( \begin{array}{c}
				36\\
				0\\
				0\\
				0\\
				0\\
				0\\
				0
			\end{array} \right)
			
		\end{Bmatrix}
	\end{equation*}
	is a Groebner basis for $\hat{R}_{G}.$

\end{Ex}

\begin{figure}[h]
	
	\begin{center}
		\begin{tikzpicture}
			\node[main node] (1) {$x \mathbb{Z}[x,y]$};
			\node[main node] (2) [below left = 3cm of 1]  {$y^2 \mathbb{Z}[x,y]$};
			\node[main node] (3) [below right = 3cm of 1] {$y \mathbb{Z}[x,y]$};
			
			\path[draw,thick]
			(1) edge node [left]{$\mathbb{Z}[x,y] /y \mathbb{Z}[x,y] $} (2)
			(2) edge node [below]{$\mathbb{Z}[x,y] /(x^2+y) \mathbb{Z}[x,y]$} (3)
			(3) edge node [right]{$ \mathbb{Z}[x,y] / x \mathbb{Z}[x,y]$} (1);

		\end{tikzpicture}
	\end{center}
	
	\caption{An edge-labeled cycle graph $(C_3,\beta)$} \label{c3xy}
\end{figure}
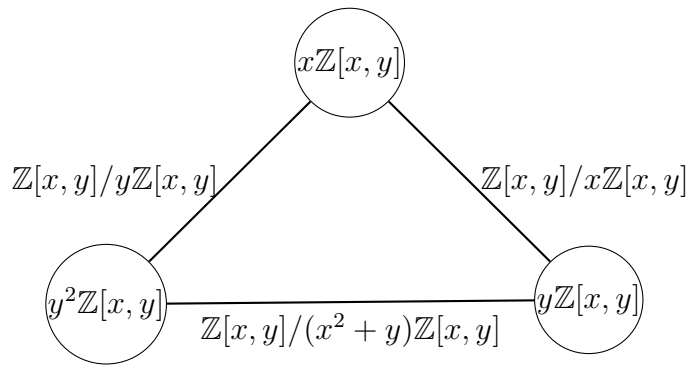

\begin{Ex}
	Let $(C_3,\beta)$ be an edge-labeled cycle graph as in Figure \ref{c3xy}. Then the GKM matrix $M_G$ is the following:	
	$$M_G=	\begin{bmatrix}
		1 & -1 & 0 & x & 0 & 0 &0 &0 &0 \\
		0 & 1 & -1& 0  & x^2+y & 0 & 0 & 0 &0\\
		1 & 0 & -1& 0  & 0 & y & 0 & 0 &0\\
		1 & 0 &  0 & 0  & 0 & 0 & x & 0 & 0\\
		0 & 1 & 0 & 0 &0 & 0 & 0 & y & 0\\
		0 & 0 & 1& 0 &0 & 0 & 0 & 0 & y^2
	\end{bmatrix}.
	$$ 	
	If we calculate $Syz(\alpha_1,\dots,\alpha_9)$ with Macaulay2 and then take the projection, we find that 
	\begin{equation*}
		\begin{Bmatrix}
			\left( \begin{array}{c}
				0\\
				0\\
				xy
			\end{array} \right), &
			\left( \begin{array}{c}
				y^2 \\
				-x^2y\\
				0
			\end{array} \right),
			&
			\left( \begin{array}{c}
				-xy^2 \\
				-xy^2\\
				-xy^2
			\end{array} \right)
		\end{Bmatrix}
	\end{equation*}
	is a Groebner basis for $\hat{R}_{G}.$

\end{Ex}


\section{A Basis Condition over $\mathbb{Z}$-module}\label{section main}
In this section, we assume that $(G, \beta)$ is an edge-labeled graph with a module $M_u=m_u\mathbb{Z}$
at each vertex $u$ and a module $M_{uv}=\mathbb{Z}/ r_{uv}\mathbb{Z}$ at each edge $uv$.  We also assume that there are quotient module homomorphisms $\varphi_u: M_u\to M_{uv}$ at each vertex $u$ incident to an adge $uv$. 

Bowden, Hagen, King and Reinders \cite{BHKR} gave a characterization of flow-up bases for generalized splines over the integers. Later,  we  gave a modification of their theorem to get a set of flow-up generators over $\mathbb{Z} /m \mathbb{Z}$ in \cite{AD}. Here, we show that the theorem also works for extending generalized splines.
\begin{Th}\label{basis cond}
	Let  $(G,\beta)$  be an edge-labeled graph.  The following are equivalent:
	\begin{itemize}
		\item The set $\{F^{(1)},\dots,F^{(n)}\}$ forms a module basis for $\hat{R}_G$.
		\item For each flow-up class $A^{(i)} =(0,\dots,0,g_{v_i},\dots,g_{v_n})$ the entry $g_{v_i}$ is a multiple of the entry $f_{v_i}^{(i)}$.
	\end{itemize}
	
\end{Th}

\begin{proof}	
	Suppose that the set $\{F^{(1)},\dots,F^{(n)}\}$ forms a basis for $\hat{R}_G$. Let   $A^{(i)} =(0,\dots,0,g_{v_i},\dots,g_{v_n})$ be a spline in  $\hat{R}_G$ with $i-1$ leading zeros. We can write $A^{(i)}$ as a linear combination of the splines $F^{(1)},\dots,F^{(n)}$ because they form a  basis. Since  $A^{(i)}$ has $i-1$ leading zeros, $A^{(i)} = a_iF^{(i)}+\cdots+a_nF^{(n)}$ for some $a_i,\dots,a_n \in \mathbb{Z}$. It follows that the $i$-th entry of $A^{(i)}$ is $g_{v_i} = a_if_{v_i}^{(i)}+a_{i+1}0+\cdots+a_n0=a_if_{v_i}^{(i)} $.
	
	Conversely, we  assume that for each flow-up class 
	$A^{(i)} =(0,\dots,0,g_{v_i},\dots,g_{v_n})$  the entry $g_{v_i}$ is a multiple of the entry $f_{v_i}^{(i)}$ of any corresponding flow-up class $F^{(i)}$, we show that  the  set $\{F^{(1)},\dots,F^{(n)}\}$ forms a module basis for $\hat{R}_G$.
	Let $A = (h_{v_1},\dots,h_{v_n})$ be an arbitrary spline with  $h_{v} \in m_v \mathbb{Z}$ at each vertex $v$ and $\beta(uv)= \mathbb{Z}/r_{uv} \mathbb{Z}$ at each edge $uv$. We will show by a finite induction that for each $j\in \{1,2,\ldots,n\}$ we can write		 
	$$A = A_j^{'} + \sum_{k=1}^{j} a_kF^{(k)}, $$
	where $A_j^{'}$ is a spline with (at least) $j$ leading zeros. The base case is when $j=1$. By assumption, we have $h_{v_1} = a_1f_{v_1}^{(1)}$ for some $a_1\in \mathbb{Z}$. Therefore, we can write  $A$ as
	$$ A= (0,h_{v_2} -a_1f_{v_2}^{(1)},\ldots,h_{v_n} -a_1f_{v_n}^{(1)}) 	+ a_1F^{(1)},$$
	where $F^{(1)}= (f_{v_1}^{(1)},\dots,f_{v_n}^{(1)})$. Here, it can be easily observed that
	$$A_1^{'} = (0,h_{v_2} -a_1f_{v_2}^{(1)},\ldots,h_{v_n} -a_1f_{v_n}^{(1)})$$ is a spline since  $h_{v_t} - a_t f_{v_t}^{(1)}\in m_{v_t} \mathbb{Z}$ for all $t=2,\dots,n$.
	Now  assume, as the induction hypothesis, that  the equality $A = A_{j-1}^{'} + \sum\limits_{k=1}^{j-1} a_kF^{(k)} $ holds for some $j\in \{2,3\ldots,n\}$, where $A_{j-1}^{'}$ is a spline with (at least) $j-1$ leading zeros.	We show that the same equality holds with $j-1$ replaced by $j$. Thus suppose that $$A= (0,\ldots,0,h_{v_j}^{'},\ldots,h_{v_n}^{'}) + \sum\limits_{k=1}^{j-1} a_kF^{(k)}.$$	
	Since  $A_{j-1}^{'}$ is a flow-up class, by assumption, we have $h_{v_j}^{'} = a_jf_{v_j}^{(j)}$ for some $a_j \in \mathbb{Z}$. It follows that $A$ can be written by
	$$ A =
	(0,\dots,0,h_{v_{j+1}}^{'} - a_jf_{v_{j+1}}^{(j)},\ldots,h_{v_n}^{'} - a_jf_{v_n}^{(j)}) + \sum\limits_{k=1}^{j} a_kF^{(k)}.$$ 
	By letting $A_j^{'}=(0,\dots,0,h_{v_{j+1}}^{'} - a_jf_{v_{j+1}}^{(j)},\ldots,h_{v_n}^{'} - a_jf_{v_n}^{(j)})$, we obtain  $A = A_j^{'} + \sum\limits_{k=1}^{j} a_kF^{(k)} $. Since $a_j \in \mathbb{Z}, h_{v_t} \in m_{v_t} \mathbb{Z}$ and $f_{v_t}^{(j)} \in m_{v_t} \mathbb{Z}$, the difference $h_{v_t} - a_t f_{v_t}^{(j)}\in m_{v_t} \mathbb{Z}$ for all $t=j+1,\dots,n$, $A_j^{'}$ is a spline. 
	In particular, we have  $A = A_n^{'} + \sum\limits_{k=1}^{n} a_kF^{(k)} $. Since $A_n^{'}=(0,\dots,0)$ is a spline with $n$ leading zeros, we obtain  $A = \sum\limits_{k=1}^{n} a_kF^{(k)} $.
	
\end{proof}

\section{A Flow-up Class for Path Graphs over $\mathbb{Z}$ }\label{Paths}

In this section, we will give an algorithm to construct a flow-up basis for the extending generalized splines on path graphs. We assume that $(P_n,\beta)$ is an edge-labeled path graph as in Figure \ref{Path n} and  $m_j$ and $r_k$ are arbitrary integers with $j=1,2, \ldots, n$ and $k=1,2, \ldots, n-1$. 

In the following theorem, we will give the characterization of the first nonzero entry of a flow-up class $F^{(i)}$ on  $(P_n,\beta)$.

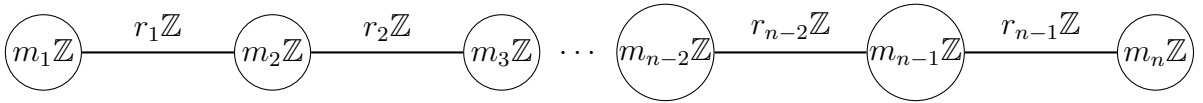
\begin{figure}[h!]
	\begin{center}

		\begin{tikzpicture}
			\node[main node] (1) {$m_1 \mathbb{Z}$};
			\node[main node] (2) [right = 2cm of 1]  {$m_2 \mathbb{Z}$};
			\node[main node] (3) [right = 2cm of 2]  {$m_3 \mathbb{Z}$};
			\node[right=.1cm of 3] {$\dots$};
			
			\node[main node] (4) [right = 1cm of 3]  {$m_{n-2} \mathbb{Z}$};
			\node[main node] (5) [right = 2cm of 4]  {$m_{n-1} \mathbb{Z}$};
			\node[main node] (6) [right = 2cm of 5]  {$m_n \mathbb{Z}$};

			\path[draw,thick]
			(1) edge node [above]{$ r_1 \mathbb{Z}$} (2)
			(2) edge node [above]{$r_2 \mathbb{Z}$} (3)

			(4) edge node [above]{$r_{n-2} \mathbb{Z}$} (5)
			(5) edge node [above]{$ r_{n-1} \mathbb{Z}$} (6);

		\end{tikzpicture}

	\end{center}
	
	\caption{An edge-labeled path graph $(P_n,\beta)$ } \label{Path n}
\end{figure}

\begin{Th}\label{multiple}
	Let $(P_n,\beta)$ be an edge-labeled path graph as in Figure \ref{Path n}. Let 	 $F^{(i)}=(0,\dots,0,f_{v_i}^{(i)}, \dots, f_{v_n}^{(i)})$ be  a flow-up class with $f_{v_i}^{(i)} \neq 0$ for $i=1,2,\ldots,n$ and $f_{v_s}^{(i)} =0$ for all $s <i$. Then  $f_{v_i}^{(i)}$ is a multiple of 
	$$[m_i,r_{i-1},(m_{i+1},r_i),(m_{i+2},r_i,r_{i+1}),\dots,(m_n,r_i,r_{i+1},\dots,r_{n-1})]$$
	for $i=1,2,\ldots,n$.
	
\end{Th}

\begin{proof}
	
	Let $F^{(i)}=(0,\dots,0,f_{v_i}^{(i)}, \dots, f_{v_n}^{(i)})$ be a flow-up class. By spline conditions, we have  $ f_{v_i}^{(i)} - 0 =f_{v_i}^{(i)}\in  r_{i-1} \mathbb{Z} ,  f_{v_j}^{(i)}- f_{v_{j+1}}^{(i)}	\in r_j \mathbb{Z}$ for all $j=i,i+1,\dots,n-1$. By solving the system of the equations together with $f_{v_k}^{(k)} \in m_k \mathbb{Z}$ at each vertex $v_k$,  we obtain inductively
	\begin{multline*}
		f_{v_{n-1}}^{(i)} - f_{v_n}^{(i) } \in r_{n-1} \mathbb{Z} \Rightarrow  f_{v_{n-1}}^{(i)} \in  f_{v_n}^{(i) } + r_{n-1} \mathbb{Z} \subset m_n \mathbb{Z}+ r_{n-1} \mathbb{Z} = (r_{n-1},m_n) \mathbb{Z} 
	\Rightarrow  f_{v_{n-1}}^{(i)} \in  [m_{n-1},(r_{n-1},m_n)] \mathbb{Z},\\
	f_{v_{n-2}}^{(i)} - f_{v_{n-1}}^{(i) } \in r_{n-2} \mathbb{Z} \quad \text{and} \quad f_{v_{n-1}}^{(i)} \in  [m_{n-1},(r_{n-1},m_n)] \mathbb{Z} \\
	\Rightarrow f_{v_{n-2}}^{(i)} \in (r_{n-2},[m_{n-1},(r_{n-1},m_n)]) \mathbb{Z}\Rightarrow f_{v_{n-2}}^{(i)} \in [m_{n-2}, (r_{n-2},[m_{n-1},(r_{n-1},m_n)])\mathbb{Z},\\
	\vdots\\
		f_{v_i}^{(i)}- f_{v_{i+1}}^{(i)} \in r_i \mathbb{Z} \quad \text{and} \quad f_{v_{i+1}}^{(i)} \in  [m_{i+1}, (r_{i+1},[\ldots,(r_{n-1},m_n)]\ldots)] \mathbb{Z} \\
		\Rightarrow f_{v_i}^{(i)} \in (r_i, [m_{i+1}, (r_{i+1},[\ldots,(r_{n-1},m_n)]\ldots)]) \mathbb{Z} \\ \Rightarrow f_{v_i}^{(i)} \in [m_i,r_{i-1}, (r_i, [m_{i+1}, (r_{i+1},[\ldots,(r_{n-1},m_n)]\ldots)])] \mathbb{Z}
	\end{multline*}
since  $ f_{v_i}^{(i)} - 0 =f_{v_i}^{(i)}\in  r_{i-1} \mathbb{Z} $ and $f_{v_i}\in m_i\mathbb Z$.
	Hence, the first non zero entry of $F^{(i)}$ can be written by
	\begin{align*}
		f_{v_i}^{(i)}&= t_i [m_i,r_{i-1},   (r_i, [m_{i+1}, (r_{i+1},[\ldots,(r_{n-1},m_n)]\ldots)])]\\
		&=t_i [m_i,r_{i-1},(m_{i+1},r_i),(m_{i+2},r_i,r_{i+1}),\dots,(m_n,r_i,r_{i+1},\dots,r_{n-1})]
	\end{align*}
	for some $t_i \in \mathbb{Z}$ where $i=1,2,\ldots,n$.

\end{proof}

\begin{Cor}\label{$g_i with n$}
	
	Let $(P_n,\beta)$ be an edge-labeled path graph as in Figure \ref{Path n} and $F=(g_1,g_2,\ldots,g_n)$   a spline. Then  $g_i$ is a multiple of 
	$$[m_i,(m_{i+1},r_i),(m_{i+2},r_i,r_{i+1}),\dots,(m_n,r_i,r_{i+1},\dots,r_{n-1})].$$

\end{Cor}

The next theorem provides an existence of minimal flow-up classes.
\begin{Th} \label{pexists}
	
	Let $(P_n,\beta)$ be an edge-labeled path graph  as in Figure \ref{Path n}.   There exists a flow-up class  $F^{(i)}=(0,\dots,0,f_{v_i}^{(i)}, \dots, f_{v_n}^{(i)})$ with the first nonzero entry  	
	$$[m_i,r_{i-1},(m_{i+1},r_i),(m_{i+2},r_i,r_{i+1}),\dots,(m_n,r_i,r_{i+1},\dots,r_{n-1})]$$
	for $i=1,2,\ldots,n$.	For $i=1$, we set $r_0=1$.

\end{Th}

\begin{proof}
	We claim the existence of the entry $f_{v_{j+1}}^{(i)}$ for all $j=i,i+1,\ldots,n-1$ where $i=1,2,\ldots,n-1$. By solving spline conditions as in the proof of Theorem \ref{multiple}, we  obtain
	\begin{align*}
		f_{v_{j+1}}^{(i)} &\equiv 0 \quad \mod m_{j+1}\\
		f_{v_{j+1}}^{(i)} &\equiv 0 \quad \mod (m_{j+2},r_{j+1})\\
		&\vdots\\
		f_{v_{j+1}}^{(i)} &\equiv 0 \quad \mod (m_{n},r_{j+1},\dots,r_{n-1})\\	
		f_{v_{j+1}}^{(i)} &\equiv f_{v_j}^{(i)} \mod r_j.
	\end{align*}
	It follows that we have a system of these equations
	\begin{align*}
		f_{v_{j+1}}^{(i)} &\equiv\quad  0 \mod [m_{j+1},(m_{j+2},r_{j+1}),\dots,(m_{n},r_{j+1},\dots,r_{n-1})]\\
		f_{v_{j+1}}^{(i)} &\equiv f_{v_j}^{(i)} \mod r_j.
	\end{align*}
	From  the Chinese Remainder Theorem  the system has a solution for $f_{v_{j+1}}^{(i)}$ if and only if 
	\[	f_{v_j}^{(i)} \equiv 0 \mod ( r_j,[m_{j+1},(m_{j+2},r_{j+1}),\dots,(m_{n},r_{j+1},\dots,r_{n-1})])\]
	or equivalently
	\begin{equation}\label{pexists:crt}
		f_{v_j}^{(i)}\equiv 0 \mod [(m_{j+1},r_j),(m_{j+2},r_j,r_{j+1}),\dots,(m_n,r_j,r_{j+1},\dots,r_{n-1})].
	\end{equation}
	To show the existence of  $f_{v_{j+1}}^{(i)}$ for all $j=i,i+1,\ldots,n-1$ where $i=1,2,\ldots,n-1$, it is sufficient to prove that  Equation \ref{pexists:crt} holds. Firstly, we will show the existence of  $f_{v_{i+1}}^{(i)}$ and move on to  $f_{v_{i+2}}^{(i)}$. Then we continue inductively to prove the existence  of $f_{v_{j+1}}^{(i)}$ for all $j=i,i+1,\ldots,n$. Start with $j=i$.
	Let  \[f_{v_i}^{(i)}= [m_i,r_{i-1},(m_{i+1},r_i),(m_{i+2},r_i,r_{i+1}),\dots,(m_n,r_i,r_{i+1},\dots,r_{n-1})].\]
	Equation \ref{pexists:crt} holds for $f_{v_{i}}^{(i)}$.	Therefore there exists a solution for $f_{v_{i+1}}^{(i)}$ by the Chinese Remainder Theorem satisfying $f_{v_{i+1}}^{(i)} \equiv 0 \mod [m_{i+1},(m_{i+2},r_{i+1}),\dots,(m_{n},r_{i+1},\dots,r_{n-1})]$. It follows that  Equation \ref{pexists:crt} holds for $f_{v_{i+1}}^{(i)}$. Therefore by the Chinese Remainder Theorem,
	$f_{v_{i+2}}^{(i)}$ exists satisfying  $f_{v_{i+2}}^{(i)} \equiv 0 \mod [m_{i+2},(m_{i+3},r_{i+2}),\dots,(m_{n},r_{i+2},\dots,r_{n-1})]$. If we continue in this way, we can prove the existence of  $f_{v_{j+1}}^{(i)}$ for all $j=i,i+1,\ldots,n-1$. Thus, there exists a flow-up set $$\{(f_{v_1}^{(1)},\ldots,f_{v_n}^{(1)}),(0,f_{v_2}^{(2)},\ldots,f_{v_n}^{(2)}),\ldots, (0,\ldots,0,f_{v_n}^{(n)})\}$$ with the given minimal first nonzero entry of $F^{(i)}$.
	
\end{proof}

Next theorem shows that minimal flow-up classes form a basis.

\begin{Th} \label{basis}
	Let $(P_n,\beta)$ be an edge-labeled path graph  as in Figure \ref{Path n}.  Let
	
	\begin{equation*}
		B = \begin{Bmatrix}
			\left( \begin{array}{c}
				f_{v_n}^{(1)}\\
				f_{v_{n-1}}^{(1)}\\
				\vdots \\
				f_{v_3}^{(1)} \\
				f_{v_2}^{(1)}\\
				f_{v_1}^{(1)}
			\end{array} \right) , &

			\left( \begin{array}{c}
				f_{v_n}^{(2)}\\
				f_{v_{n-1}}^{(2)}\\
				\vdots \\
				f_{v_3}^{(2)} \\
				f_{v_2}^{(2)}\\
				0
			\end{array} \right) , &

			\ldots  & ,

			\left( \begin{array}{c}
				f_{v_n}^{(i)}\\
				\vdots\\
				f_{v_i}^{(i)}\\
				0\\
				\vdots \\
				0
			\end{array} \right) , &
			
			\ldots & ,
			
			\left( \begin{array}{c}
				f_{v_n}^{(n)}\\
				0\\
				0\\
				\vdots \\
				0 \\
				0
			\end{array} \right) 
		\end{Bmatrix}
	\end{equation*}
	be a flow-up class set	where each $F^{(i)} \in \hat{\mathcal{F}_i}$ is minimal with the first nonzero entry	
	$$[m_i,r_{i-1},(m_{i+1},r_i),(m_{i+2},r_i,r_{i+1}),\dots,(m_n,r_i,r_{i+1},\dots,r_{n-1})]$$
	for $i=1,2,\ldots,n$. Then $B$ forms a module basis over $\mathbb{Z}$.
\end{Th}

\begin{proof}
	It follows from Theorem \ref{basis cond}.
\end{proof}

\begin{Nt*} By Corollary {\ref{$g_i with n$}}, we have 	$f_{v_i}^{(k)}=l_i^{(k)}s_i^{(k)}$
	where $s_i^{(k)} = [m_i, (r_i,[m_{i+1},\ldots,(m_n,r_{n-1})]\ldots)]$. It follows that
	by Chinese Remainder Theorem (or by Proposition \ref{Bowden}) there is a solution of  the system 
	\begin{align*}
		f_{v_t}^{(k)}&\equiv 0 \quad\,\, \mod s_t^{(k)}\\
	f_{v_t}^{(k)} &\equiv f_{v_{t-1}}^{(k)} \mod r_{t-1}
	\end{align*}
	 where $s_t^{(k)} = [m_t, (r_t,[m_{t+1},\ldots,(m_n,r_{n-1})]\ldots)]$
	with
	$$ l_t^{(k)}\equiv \bigg(\bigg(\frac{s_t^{(k)}}{(s_t^{(k)},r_{t-1})}\bigg)^{-1} \mod \bigg(\frac{r_{t-1}}{(s_t^{(k)},r_{t-1})}\bigg)\bigg) \frac{ f_{v_{t-1}}^{(k)}}{(s_t^{(k)},r_{t-1})} \mod \bigg(\frac{r_{t-1}}{(s_t^{(k)},r_{t-1})}\bigg). $$

\end{Nt*}

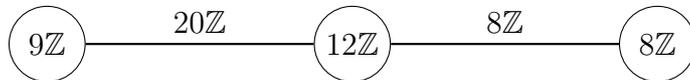
\begin{figure}[h]
	\begin{center}

		\begin{tikzpicture}
			\node[main node] (1) {$9 \mathbb{Z}$};
			\node[main node] (2) [right = 3cm of 1]  {$12 \mathbb{Z}$};
			\node[main node] (3) [right = 3cm of 2] {$8 \mathbb{Z}$};
			
			\path[draw,thick]
			(1) edge node [above]{$ 20 \mathbb{Z}$} (2)
			(2) edge node [above]{$ 8 \mathbb{Z}$} (3);

		\end{tikzpicture}

	\end{center}
	
	\caption{Example of an edge-labeled path graph $(P_3,\beta)$} \label{ex}

\end{figure}

\begin{Ex}
	
	Consider the edge-labeled path graph $(P_3,\beta)$ as in Figure \ref{ex}. Then a
	basis B can be obtained by using the note above and Theorem \ref{basis} as follows:
	\begin{gather*}
	f_{v_1}^{(1)}=  36,\\
	f_{v_2}^{(1)}=l_2^{(1)}s_2^{(1)} \text{ with } l_2^{(1)}\equiv (6^{-1} \mod 5 ) 9 \equiv 4 \mod 5, \\
	f_{v_3}^{(1)}=l_3^{(1)}s_3^{(1)}  \text{ with }  l_3^{(1)} \equiv  (1^{-1} \mod 1 ) 12 \mod 1 \text{ or }l_3^{(1)}= 0; \\	
	f_{v_1}^{(1)}=0, f_{v_2}^{(2)}= 120,\\
	f_{v_3}^{(2)}=l_3^{(2)}s_3^{(2)}  \text{ with }  l_3^{(2)}\equiv (1 ^{-1} \mod 1)15 \mod 1  \text{ or } l_3^{(2)}=0;  \\
	f_{v_1}^{3}=f_{v_2}^{(3)}=0, f_{v_3}^{(3)}=  8.
	\end{gather*}
A flow-up basis $B$ is given by

	\begin{equation*}
		B = \begin{Bmatrix}
			\left( \begin{array}{c}
				0 \\
				96 \\
				36
			\end{array} \right) , &

			\left( \begin{array}{c}
				0  \\
				120 \\
				0 \\
			\end{array} \right), &
			
			\left( \begin{array}{c}
				8 \\
				0 \\
				0 \\
			\end{array} \right)
			
		\end{Bmatrix}.
	\end{equation*}

\end{Ex}

\section{Longest Paths Technique}

In this section, we develop a technique for constructing a flow-up basis on arbitrary graphs based on the argument used in paths in Section \ref{Paths}, called a longest paths technique.

We assume that  $(G, \beta)$ is an edge-labeled graph together with a $\mathbb{Z}$-module $M_{v_i}= m_i \mathbb{Z}$ at each vertex  and  $\beta(v_iv_j)= \mathbb{Z}/ r_{ij} \mathbb{Z}$ at each edge.  For simplicity, we associate $r_{ij} \mathbb{Z}$ to each edge label $ l_{ij}=\beta(v_iv_j)= \mathbb{Z} / r_{ij} \mathbb{Z} $.

\begin{Def}
	Let  $(G,\beta)$  be an edge-labeled graph.   
	A \emph{trail} is an ordered sequence of vertices and edges $v_j,e_{j-1},v_{j-1},\ldots,v_{i+1},e_{i},v_i$ in which no edge is repeated.
	A trail from $v_j$ to $v_i$ is a path graph denoted by $P_{ji}$. We will use a notation $P_{ji}= (v_je_{j-1}v_{j-1}\cdots v_{i+1}e_{i}v_i)$  or  $(m_jr_{j-1}m_{j-1}\cdots m_{i+1}r_{i}m_i)$  corresponding to the labels of vertices and edges here. For a fixed $i$ we consider all trails $P_{ki}$ from $v_k$ to $v_i$ for all $k$. A longest path $P_{ji}$ from $v_j$ to $v_i$  is defined a trail of maximum length of all trails $P_{ki}$ under the strict inclusion $ P_{ki} \subset P_{li}$ for all $k,l$. When $j$ is not important we say that a longest path $P_{ji}$  from $v_j$ to $v_i$ is a longest path $P_{ji}$ of $v_i$. The set of longest paths $P_{ji}$ of $v_i$ is denoted by $\mathcal{P}^{i}$.
\end{Def}

We denote the greatest common divisor of the edge labels on   $P_{ji}$ by $(p^{(i,j)})$. The set of greatest common divisors $(p^{(i,j)})$ on all trails $P_{ij}$  is represented as $\{(p^{(i,j)})\}.$

\begin{Th}\label{longmult}
	
	Let $(G,\beta)$ be an edge-labeled graph and  $H=(h_1,h_2,\ldots,h_n)$   a spline.  Let  $\mathcal{P}^{i}$ be the set of the longest paths $P_{ji}=(m_jr_{j-1}m_{j-1}\cdots m_{i+1}r_im_i)$  and a spline on each longest path $P_{ji}$ together with the $i$-th entry  divisible by
	$$ [m_i,(m_{i+1},r_i),(m_{i+2},r_i,r_{i+1}),\dots,(m_j,r_i,r_{i+1},\dots,r_{j-1})].$$ Then $h_i$ is determined by the $i$-th entries of $H$ on each longest path. In other words, $h_i$ is divisible by  the least common multiple of all the $i$-th entries.

\end{Th}
\begin{proof}
	It follows from Corollary \ref{$g_i with n$}.

\end{proof}

\begin{Cor}
	Let $(G,\beta)$ be an edge-labeled graph and   $H=(h_1,h_2,\ldots,h_n)$   a spline. Let $P_{ji}$ be a trail	
	and $\{(p^{(i,j)})\}$  the set of greatest common divisors of the edge labels on $P_{ji}$.
	Then  \[[m_i,\{(m_j,[ \{(p^{(i,j)})\}]) \mid \forall j\neq i \}] \] divides $h_i$  for all $j$. 	
\end{Cor}

\begin{Th}\label{minimal}	Let  $(G,\beta)$  be an edge-labeled graph.	 Let  $\mathcal{P}^{1}$ be the set of the longest paths $P_{j1}$ of $v_1$. Then there exists a minimal spline $F=(f_{v_1},f_{v_2},\dots,f_{v_j})$ on each longest path $P_{j1}=(m_jr_{j-1}m_{j-1}\cdots m_2r_{1}m_1)$ with $f_{v_1}=[m_1,(m_2,r_1),(m_3,r_1,r_2),\dots,(m_j,r_1,r_2,\dots,r_{j-1})]$.
	
\end{Th}
\begin{proof} It is obvious by Theorem \ref{pexists}.

\end{proof}

\begin{Def}[Zero Vertex Reduction]\label{graphre}
	Let  $(G,\beta)$  be an edge-labeled graph and $F=(f_{v_1},\ldots,f_{v_n})$ a spline with $f_{v_i}=0$ for  a unique $i$. We define a new graph $(G_{zred},\beta_{zred})$ associated to the zero-labeled vertex $v_i$ is an edge-labeled graph defined by removing the vertex $v_i$  and all the incident edges $v_iv_{i_t}$ of $G$ but keep the adjacent vertices $v_{i_t}$ by relabeling $[m_{v_i},r_{v_iv_{i_t}}]\mathbb{Z}$. We call such a graph $(G_{zred},\beta_{zred})$ a  reduced graph associated to a zero-labeled vertex $v_i$.
\end{Def}

This definition can be repeated on zero-labeled vertices successively. This graph is not necessarily connected.

To determine the smallest first nonzero entry of a flow-up class $F^{(i)}=(0,\dots,0,f_{v_i}^{(i)}, \dots, f_{v_n}^{(i)})$ in $\hat R_{(G,\beta)}$ is the same as to determine the smallest first nonzero entry of a flow-up class $F^{(i)}=(f_{v_i}^{(i)}, \dots, f_{v_n}^{(i)})$ in $\hat{R}_{(G_{zred},\beta_{zred})}$. It follows the proposition and the theorem below.

\begin{Prop}\label{remove}
	Let  $(G,\beta)$  be an edge-labeled graph and $F^{(i)}=(0,\dots,0,f_{v_i}^{(i)}, \dots, f_{v_n}^{(i)})$  a flow-up class with $f_{v_i}^{(i)} \neq 0$ for $i>1$ and $f_{v_s}^{(i)} =0$ for all $s <i$.  Let  $(G_{zred},\beta_{zred})$ be a reduced graph.  Then, 	$$(0,\dots,0,f_{v_i}^{(i)}, \dots, f_{v_n}^{(i)}) \in \hat{R}_G \Leftrightarrow (f_{v_i}^{(i)}, \dots, f_{v_n}^{(i)}) \in \hat{R}_{G_{zred}}.$$	
	
\end{Prop}

\begin{proof}
	By definition we know that $f_{v_j} \in m_{v_j} \mathbb{Z}$ for all $j$. Assume that $v_l$ is adjacent to $v_{l_1},\ldots, v_{l_k}$ for some $l \ge i$ where $l_1,\ldots, l_k <i$. Then, we have
	\begin{center}
		$ (0,\dots,0,f_{v_i}^{(i)}, \dots, f_{v_n}^{(i)}) \in \hat{R}_G \Leftrightarrow f_{v_l}^{(i)} \equiv 0 \mod m_{v_l}\text { and  }f_{v_l}^{(i)} \equiv 0 \mod r_{v_lv_{l_t}} 
		\text{ for } t=1,\ldots, k$ 
		$\Leftrightarrow f_{v_l}^{(i)} \in [m_{v_l},r_{v_lv_{l_1}},\ldots,r_{v_lv_{l_k}}] \mathbb{Z} \Leftrightarrow (f_{v_i}^{(i)}, \dots, f_{v_n}^{(i)}) \in \hat{R}_{G_{zred}}.$
	\end{center}

\end{proof}

\begin{Th}\label{flow-uplcm}
	Let  $(G,\beta)$  be an edge-labeled graph and  $F^{(i)}=(0,\dots,0,f_{v_i}^{(i)}, \dots, f_{v_n}^{(i)})$  a flow-up class with $f_{v_i}^{(i)} \neq 0$ for $i>1$ and $f_{v_s}^{(i)} =0$ for all $s <i$. Let $(G_{zred},\beta_{zred})$ be a reduced graph.  Let  $\mathcal{P}^{i}$ be the set of the longest paths of $v_i$ on $(G_{zred},\beta_{zred})$ and    $F_k=(f_{kv_i},f_{kv_{j_1}},\dots,f_{kv_{j_t}})$   a minimal spline with $f_{kv_i} \neq 0$ on $(P_{ji},\beta_{zred_{ji}})$ where $P_{ji} \in \mathcal{P}^{i}$ and $\beta_{zred_{ij}}=\beta_{zred}{\mid{P_{ji}}}$.
	Then $f_{v_i}^{(i)}$ is a multiple of  the least common multiple of all $f_{kv_i}$.	
\end{Th}
\begin{proof}
	It follows from Theorem \ref{longmult} and Proposition \ref{remove}.
\end{proof}

\noindent

\begin{figure}[H]
	
	\begin{center}
		\begin{tikzpicture}
			
			\node[main node] (1) {$m_1 \mathbb{Z}$};
			\node[main node] (2) [right = 2cm of 1]  {$m_2 \mathbb{Z}$};
			\node[main node] (3) [right = 2cm of 2]  {$m_3 \mathbb{Z}$};
			\node[main node] (4) [above left = 3cm of 3]  {$m_4 \mathbb{Z}$};
			\node[main node] (5) [right = 2cm of 3]  {$m_5 \mathbb{Z}$};
			\node[main node] (6) [right = 2cm of 4]  {$m_6 \mathbb{Z}$};

			\path[draw,thick]
			(1) edge node [above]{$r_1 \mathbb{Z}$} (2)
			(2) edge node [above]{$r_2 \mathbb{Z}$} (3)
			(3) edge node [right]{$   r_3 \mathbb{Z} $} (4)
			(2) edge node [left]{$ r_4  \mathbb{Z}  $} (4)
			(3) edge node [above]{$ r_5 \mathbb{Z}$} (5)
			(4) edge node [above]{$  r_6 \mathbb{Z}$} (6);

		\end{tikzpicture}
	\end{center}
	
	\caption{An edge-labeled graph $(G,\beta)$} \label{longest}
\end{figure}
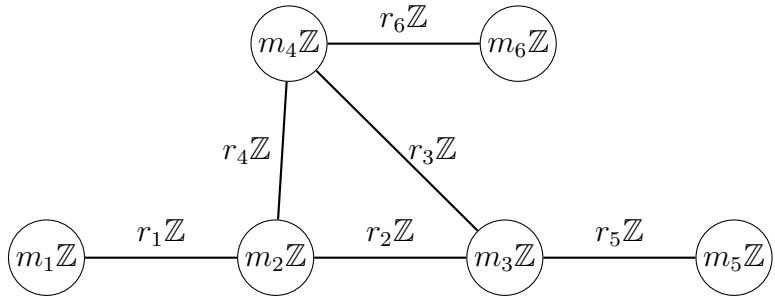	

\begin{figure}[h!]
	\centering
	
	\begin{subfigure}[b]{0.45\textwidth}
		\centering
		\begin{tikzpicture}
			\node[main node] (1)[right=2 cm] {$[m_2,r_1] \mathbb{Z}$};
			\node[main node] (2) [right = 2cm of 1]  {$m_3 \mathbb{Z}$};
			\node[main node] (3) [above left = 2cm of 2]  {$m_4 \mathbb{Z}$};
			\node[main node] (4) [right = 2cm of 2]  {$m_5 \mathbb{Z}$};
			\node[main node] (5) [right = 2cm of 3]  {$m_6 \mathbb{Z}$};
			\path[draw,thick]
			(1) edge node [below]{$r_2 \mathbb{Z} $} (2)
			(2) edge node [right]{$ r_3 \mathbb{Z}$} (3)
			(1) edge node [left]{$ r_4 \mathbb{Z} $} (3)
			(2) edge node [below]{$  r_5 \mathbb{Z}$} (4)
			(3) edge node [above]{$ r_6\mathbb{Z} $} (5);
		\end{tikzpicture}
		\caption{ $(G_2,\beta_2)$}
	\end{subfigure}
	\hfill
	\begin{subfigure}[b]{0.45\textwidth}
		\centering
		\begin{tikzpicture}
			\node[main node] (6)  {$[m_3,r_2] \mathbb{Z}$};
			\node[main node] (7) [above = 1cm of 6]  {$[m_4,r_4] \mathbb{Z}$};
			\node[main node] (8) [right = 2cm of 6]  {$m_5 \mathbb{Z}$};
			\node[main node] (9) [right = 2cm of 7]  {$m_6 \mathbb{Z}$};
			\path[draw,thick]
			(6) edge node [left]{$  r_3 \mathbb{Z}   $} (7)
			(6) edge node [below]{$  r_5 \mathbb{Z}$} (8)
			(7) edge node [above]{$  r_6 \mathbb{Z}$} (9);
		\end{tikzpicture}
		\caption{$(G_3,\beta_3)$}
	\end{subfigure}
	
	\vspace{1cm} 
	
	\begin{subfigure}[t]{0.3\textwidth}
		\centering
		\begin{tikzpicture}[baseline={(current bounding box.north)}]
			\node[main node] (10)  {$[m_4,r_3,r_4] \mathbb{Z}$};
			\node[main node] (11) [below right = 1cm of 10] {$[m_5,r_5] \mathbb{Z}$};
			\node[main node] (12) [right = 2cm of 10]  {$m_6 \mathbb{Z}$};
			\path[draw,thick] (10) edge node [above]{$ r_6 \mathbb{Z}$} (12);
		\end{tikzpicture}
		\caption{$(G_4,\beta_4)$}
	\end{subfigure}
	\hfill
	\begin{subfigure}[t]{0.3\textwidth}
		\centering
		\begin{tikzpicture}[baseline={(current bounding box.north)}]
			\node[main node] (10)  {$[m_6,r_6] \mathbb{Z}$};
			\node[main node] (11) [below  = 1cm of 10] {$[m_5,r_5] \mathbb{Z}$};
		\end{tikzpicture}
		\caption{$(G_5,\beta_5)$}
	\end{subfigure}
	\hfill
	\begin{subfigure}[t]{0.3\textwidth}
		\centering
		\begin{tikzpicture}[baseline={(current bounding box.north)}]
			\node[main node] (12)  {$[m_6,r_6] \mathbb{Z}$};
		\end{tikzpicture}
		\caption{$(G_6,\beta_6)$}
	\end{subfigure}
	\caption{Edge-labeled graphs $(G_i,\beta_i)$}
	\label{subg}
\end{figure}
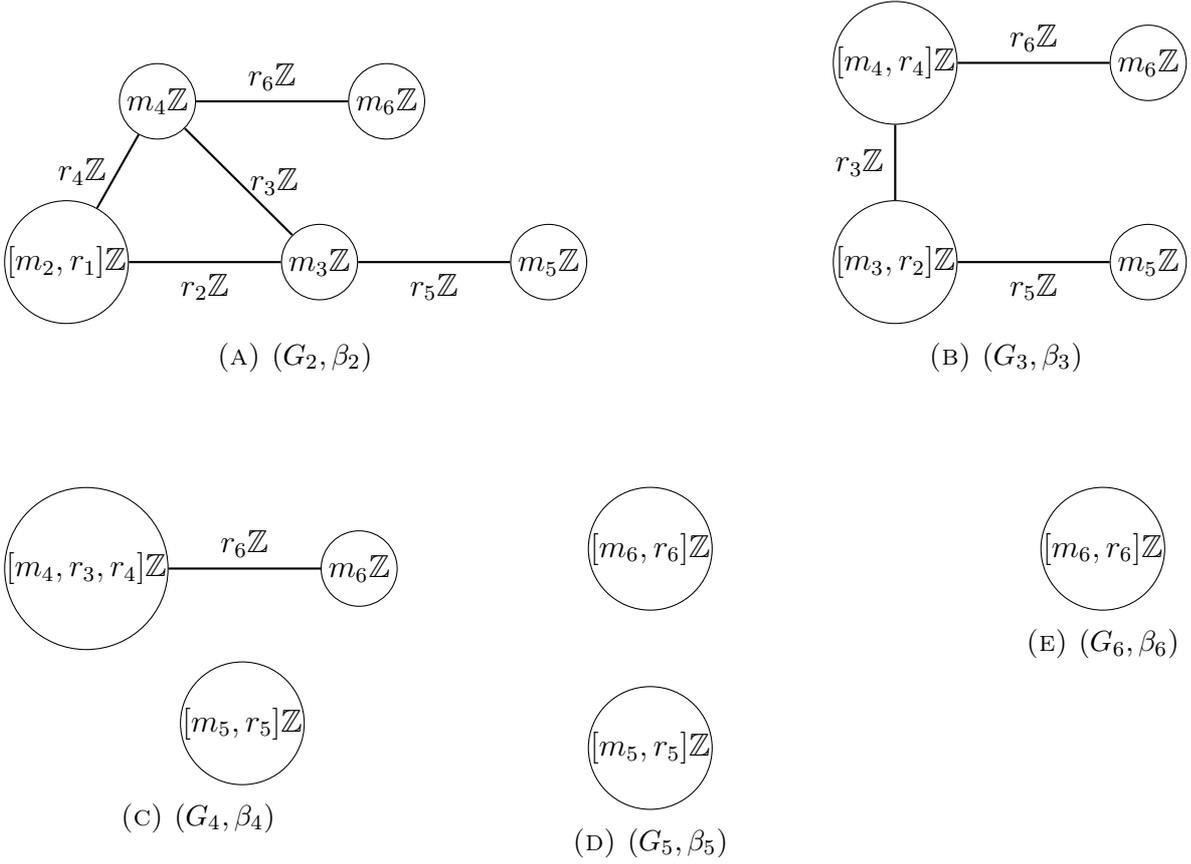

\begin{Ex}\label{exlong}
	
Consider the edge-labeled graph $(G,\beta)$ as shown in Figure \ref{longest}. Let $F^{(i)}$ be a flow-up class on $(G,\beta)$, where $f_{v_i}^{(i)} \neq 0$ for each $i$ with $1 \le i \le 6$, and $f_{v_s}^{(i)} = 0$ for all $s < i$. We can determine a necessary condition  $f_{v_i}^{(i)}$ using longest paths. For $i=1$, set $G_1=G$. The set of longest paths on $G_1$ is $\mathcal{P}^{1} = \{P^1_{61}, P^2_{61}, P^1_{51}, P^2_{51}\}$ where 
$P^1_{61}= (m_6r_6m_4r_4m_2r_1m_1)$,
$P^2_{61}= (m_6r_6m_4r_3m_3r_2m_2r_1m_1)$,
$P^1_{51}= (m_5r_5m_3r_2m_2r_1m_1)$,
$P^2_{51}= (m_5r_5m_3r_3m_4r_4m_2r_1m_1).$

	Let $a_1,a_2,a_3,a_4$  be the first nonzero entries of a minimal flow-up class on each longest path $(P^1_{61},\beta^{1}_{61}),(P^2_{61},\beta^{2}_{61}),(P^1_{51},\beta^{1}_{51}),(P^2_{51},\beta^{2}_{51}),$ respectively.
	By Theorem \ref{longmult}, 
	$f_{v_1}^{(1)}$ is a multiple of $[a_1,a_2,a_3,a_4]$ where 
	\begin{align*}
		a_1 = &[m_1,(m_2,r_1),(m_4,r_1,r_4),(m_6,r_1,r_4,r_6)] \\
		a_2= &[m_1,(m_2,r_1),(m_3,r_1,r_2),(m_4,r_1,r_2,r_3),(m_6,r_1,r_2,r_3,r_6)]\\
		a_3= &[m_1,(m_2,r_1),(m_3,r_1,r_2),(m_5,r_1,r_2,r_5)]\\
		a_4= &[m_1,(m_2,r_1),(m_4,r_1,r_4),(m_3,r_1,r_3,r_4),(m_5,r_1,r_3,r_4,r_5)].
	\end{align*}
	
	For any $2\le i \le 6$, we want to find a similar condition for  $f_{v_i}^{(i)}$. Since $f_{v_s}^{(i)} = 0$ for all $s < i$, we first remove the vertices with zero labels, resulting in new graphs $(G_i, \beta_i)$ for each $i$ with $2 \le i \le 6$.
	For each $2\le i \le 4$, we have the following sets of longest paths on $(G_i, \beta_i)$: 
	For $i=2$, $ \mathcal{P}^{2}= \{P^1_{62},P^2_{62},P^1_{52},P^2_{52}\}$  on $G_2$ where  $P^1_{62}= (m_6r_6m_4r_4m^{\prime}_2)$, $P^2_{62}= (m_6r_6m_4r_3m_3r_2m^{\prime}_2)$, $P^1_{52}=  (m_5r_5m_3r_2m^{\prime}_2) $ and $P^2_{52} =  (m_5r_5m_3r_3m_4r_4m^{\prime}_2)$  together with $m^{\prime}_2=[m_2,r_1]$. For $i=3$, $ \mathcal{P}^{3}= \{P^1_{63},P^1_{53}\}$ on $G_3$ where	$P^1_{63}= (m_6r_6m^{\prime}_4r_3m^{\prime}_3) $,
	$P^1_{53}=  (m_5r_5m^{\prime}_3)$ together with $m^{\prime}_3=[m_3,r_2]$ and $m^{\prime}_4=[m_4,r_4]$. For $i=4$, 	there is only one path $P^1_{64}=  (m_6r_6m^{\prime \prime}_4)$ on $G_4$ together with $m^{\prime \prime}_4=[m_4,r_3,r_4]$. Therefore
by Theorem \ref{longmult}, we obtain $f_{v_2}^{(2)}, f_{v_3}^{(3)}, f_{v_4}^{(4)}$ are divisible by
$[m_2,r_1,(m_4,r_4),(m_6,r_4,r_6),(m_4,r_2,r_3),(m_6,r_2,r_3,r_6),(m_3,r_2),(m_5,r_2,r_5),(m_3,r_3,r_4),(m_5,r_3,r_4,r_5)],$
$[m_3,r_2,(m_5,r_5),(m_4,r_3),(m_6,r_3,r_6),(r_3,r_4)]$ and
		$[m_4,r_3,r_4,(m_6,r_6)]$ respectively.
	\normalsize
	Moreover, for $i=5, 6$,  $(G_5,\beta_5)$ is a union of null graphs and  $(G_6,\beta_6)$ is a null graph. Hence
	
	$$	f_{v_5}^{(5)}= [m_5,r_5]  \text{  and  }
		f_{v_6}^{(6)}= [m_6,r_6].$$

\end{Ex}

\begin{Lemma}\label{divisible k_1}
	Let $(G,\beta)$ be an edge-labeled graph. Let $P_{ji}$ be a trail	
	and $\{(p^{(i,j)})\}$  the set of greatest common divisors of the edge labels on $P_{ji}$. Let \[k_i=[m_i,\{(m_j,[ \{(p^{(i,j)})\}]) \mid \forall j\neq i \}].\]
	Then  $k_i$ is a multiple of  $([\{(p^{(i,t)})\}],k_{t})$ for all $i<t$.	
\end{Lemma}

\begin{proof}
 For a fixed $t$, consider any $i<t$. Then we can rewrite 
		\begin{align*}
		k_t &= [m_t,(m_i,[ \{(p^{(i,t)})\}]),\{(m_l,[ \{(p^{(i,l)})\}]) \mid \forall l\neq i,t \}].
	\end{align*} If we substitute $k_t$	
	into  $	([\{(p^{(i,t)})\}],k_{t})$, we obtain 
	\begin{align}
		([\{(p^{(i,t)})\}],k_{t})
		&= ([\{(p^{(i,t)})\}],[m_t,(m_i,[ \{(p^{(i,t)})\}]),\{(m_l,[ \{(p^{(i,l)})\}]) \mid \forall l\neq i,t \}])
		\label{k_i 2}.
	\end{align}
	Rearranging Equation  \ref{k_i 2}, we get
	\begin{align*}
		\bigg[\bigg(m_t,[\{(p^{(i,t)})\}]\bigg),\bigg(m_i,[ \{(p^{(i,t)})\}],[\{(p^{(i,t)})\}]\bigg),\bigg(\{(m_l,[ \{(p^{(i,l)})\}]) \mid \forall l\neq i,t \}, [\{(p^{(i,t)})\}]\bigg) \bigg].
	\end{align*}
	It is straightforward to see that $k_i$ is a multiple of both $(m_t,[\{(p^{(i,t)})\}])$ and $(m_i,[ \{(p^{(i,t)})\}])$. We now aim to demonstrate that $k_i$ is also divisible by $$\bigg(\{(m_l,[ \{(p^{(i,l)})\}]) \mid \forall l\neq i,t \}, [\{(p^{(i,t)})\}]\bigg).$$
	Since each $(m_l,[ \{(p^{(i,l)})\}])$ with $l\neq i $ is itself a multiple of this expression, it follows that $k_i$ is divisible by $\bigg(\{(m_l,[ \{(p^{(i,l)})\}]) \mid \forall l\neq i,t \}, [\{(p^{(i,t)})\}]\bigg)$.
\end{proof}

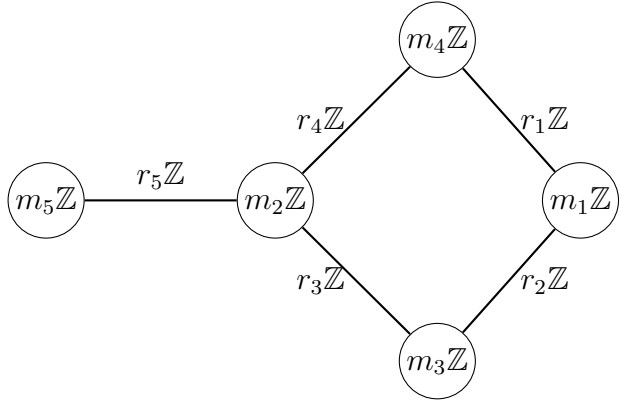
\begin{figure}[h]
	
	\begin{center}
		\begin{tikzpicture}
			
			\node[main node] (1) {$m_5 \mathbb{Z}$};
			\node[main node] (2) [right = 2cm of 1]  {$m_2 \mathbb{Z}$};
			\node[main node] (3) [below right = 2cm of 2]  {$m_3 \mathbb{Z}$};
			\node[main node] (4) [above right = 2cm of 2]  {$m_4 \mathbb{Z}$};
			\node[main node] (5) [right = 3cm of 2]  {$m_1 \mathbb{Z}$};

			\path[draw,thick]
			(1) edge node [above]{$ r_5\mathbb{Z} $} (2)
			(2) edge node [left]{$r_3 \mathbb{Z}$} (3)
			(3) edge node [right]{$ r_2 \mathbb{Z}$} (5)
			(2) edge node [left]{$ r_4  \mathbb{Z} $} (4)
			(4) edge node [right]{$ r_1 \mathbb{Z}$} (5);

		\end{tikzpicture}
	\end{center}
	
	\caption{An edge-labeled graph $(G,\beta)$} \label{path}
\end{figure}

\begin{Ex}
	Consider the edge-labeled graph $(G,\beta)$ as shown in Figure \ref{path}.	For each trail  $P_{ji}$, let $\{(p^{(i,j)})\}$ denote the set of greatest common divisors of the edge labels along that trail. Define $k_i=[m_i,\{(m_j,[ \{(p^{(i,j)})\}]) \mid \forall j\neq i \}].$ Then we have:
	\begin{align*}
		k_1=& [m_1,\{(m_j,[ \{(p^{(1,j)})\}]) \mid \forall j\neq 1\}]\\
		=&[m_1,(m_2,[(r_1,r_4),(r_2,r_3)]),(m_3,[r_2,(r_1,r_3,r_4)]),(m_4,[r_1,(r_2,r_3,r_4)]),(m_5,[(r_1,r_4,r_5),(r_2,r_3,r_5)])] \\
	\text{and }	\\
		k_2=&  [m_2,\{(m_j,[ \{(p^{(2,j)})\}]) \mid \forall j\neq 2\}]\\
		=& [m_2,(m_1,[(r_1,r_4),(r_2,r_3)]), (m_3,[r_3,(r_1,r_2,r_4)]),(m_4,[r_4,(r_1,r_2,r_3)]),(m_5,r_5)].
	\end{align*} 
Substituting $k_2$ into $([(p_1^{(2,1)}),(p_2^{(2,1)})],k_2)$,  where $(p_1^{(2,1)})=(r_1,r_4),(p_2^{(2,1)})=(r_2,r_3)$, yields:
\begin{align}
	\begin{split}
		([(p_1^{(2,1)}),(p_2^{(2,1)})],k_2)&= ([(r_1,r_4),(r_2,r_3)],[m_2,\{(m_j,[ \{(p^{(2,j)})\}]) \mid \forall j\neq 2\}])\\
		& = [(m_2,r_1,r_4),(m_1,r_1,r_4),(m_1,r_1,r_2,r_3,r_4),(m_3,r_1,r_3,r_4),(m_3,r_1,r_2,r_4),\\
		&(m_4,r_1,r_4),(m_4,r_1,r_2,r_3,r_4), (m_5,r_1,r_4,r_5),(m_2,r_2,r_3),(m_1,r_1,r_2,r_3,r_4),(m_1,r_2,r_3),\\
		&(m_3,r_2,r_3),(m_3,r_1,r_2,r_3,r_4),
		(m_4,r_2,r_3,r_4),(m_4,r_1,r_2,r_3),(m_5,r_2,r_3,r_5)].	
	\end{split}
	\label{list}
\end{align}	
Now, given that 	\small\begin{align*}
	k_1=&[m_1,\{(m_j,[ \{(p^{(1,j)})\}]) \mid \forall j\neq 1\}]\\
=&	[m_1,(m_2,r_1,r_4),(m_2,r_2,r_3),(m_3,r_2),(m_3,r_1,r_3,r_4),(m_4,r_1),(m_4,r_2,r_3,r_4),(m_5,r_1,r_4,r_5),(m_5,r_2,r_3,r_5)],
\end{align*}\normalsize we can clearly see that since $m_1$ is divisible by	$(m_1,r_1,r_4),(m_1,r_1,r_2,r_3,r_4)$ and $(m_1,r_2,r_3)$;  $(m_3,r_2)$ is divisible by $(m_3,r_1,r_2,r_4),(m_3,r_2,r_3),(m_3,r_1,r_2,r_3,r_4)$; $(m_4,r_1)$ is divisible by  $(m_4,r_1,r_4),$ $(m_4,r_1,r_2,r_3,r_4)$ and $(m_4,r_1,r_2,r_3)$ in the list referred to as \ref{list}, $k_1$ is divisible by
	$([(p_1^{(2,1)}),(p_2^{(2,1)})],k_2).$
\end{Ex}

The following theorem establishes the existence of a flow-up class $F^{(1)}=(f_{v_1}^{(1)},f_{v_2}^{(1)},\dots,f_{v_n}^{(1)})$ on an arbitrary graph.

\begin{Th}\label{exist}
	Let $(G,\beta)$ be an edge-labeled graph and   $\mathcal{P}^{1}$  the set of the longest paths $P_{j1}=(m_jr_{i_{t+1}}m_{i_t}r_{i_{t}}\dots r_{i_2}m_{i_1}r_{i_1}m_1)$ of $v_1$. Let    $F_k=(f_{kv_1},f_{kv_{i_1}},\dots,f_{kv_{i_t}},f_{kv_j})$  be a minimal spline with $f_{kv_1} \neq 0$ on $(P_{j1},\beta_{j1})$ where $P_{j1} \in \mathcal{P}^{1}$ and $\beta_{j1}=\beta{\mid{P_{j1}}}$.
	Then there exists a flow-up class $F^{(1)}=(f_{v_1}^{(1)},f_{v_2}^{(1)},\dots,f_{v_n}^{(1)})$ with the first nonzero entry $f_{v_1}^{(1)}$ equal to  the least common multiple of all $f_{kv_1}$.
\end{Th}

\begin{proof}
	Let $f_{v_1}^{(1)}$ be defined as the least common multiple of all terms $f_{kv_1}$.
We assert that corresponding entries $f_{v_{i}}^{(1)}$ of $F^{(1)}$ exist for all $i=2,\dots,n$. To demonstrate this, we begin by establishing the existence of $f_{v_{2}}^{(1)}$, followed by  $f_{v_{3}}^{(1)}$. Following this pattern, we continue inductively to prove the existence of  $f_{v_{i}}^{(1)}$ for all $i=4,\dots,n$. 
	Let  $H=(h_1,h_2,\ldots,h_n)$ be a spline on $(G,\beta)$ and  $P_{ji} \in \mathcal{P}^{i}$  a longest path for $i<j$.  For simplicity, we write a longest path   $P_{ji}= (m_jr_{j_{s+1}}m_{j_s}r_{j_{s}}\dots m_{j_1}r_{j_1}m_i) \in \mathcal{P}^{i}$ for $i<j$. Then, by Theorem \ref{longmult},  $h_i$ is divisible by the least common multiple $k_i$ of  $$[m_i,(m_{j_1},r_{j_1}),(m_{j_2},r_{j_1},r_{j_2}),\dots,(m_j,r_{j_1},r_{j_2},\dots,r_{j_{s+1}})]$$ for each longest path $P_{ji}$. Thus, $h_i=k_it_i$ for some $t_i$. 
	Assume that  the set 
	$\mathcal{P}^{2}=\{P_{2i}^{(1)},P_{2i}^{(2)},\dots  \}$ represents the collection of longest paths.
	Let $\{(p_1^{(2,1)}),(p_2^{(2,1)}), \dots,(p_t^{(2,1)})\}$ be the set of greatest common divisors of the edge labels on $P_{21}$ for some $t$. We can then analyze a corresponding system of equations, denoted as
	\begin{align*}
		f_{v_{2}}^{(1)} &\equiv	f_{v_{1}}^{(1)} \mod [\{(p_l^{(2,1)}) \mid 1\le l \le t\}] \\
		f_{v_{2}}^{(1)} &\equiv 0  \quad \mod k_2.
	\end{align*}
	By Theorem \ref{CRT} (CRT),  the system has a solution for $f_{v_{2}}^{(1)}$  if and only if \[f_{v_{1}}^{(1)} \equiv 0 \mod ([\{(p_l^{(2,1)}) \mid 1\le l \le t\}],k_2).\]   By Lemma \ref{divisible k_1}, we know that $k_1$ is a multiple of $([\{(p_l^{(2,1)}) \mid 1\le l \le t\}],k_2)$. This  guarantees the existence of a solution for $f_{v_{2}}^{(1)}$,  which can be written as $f_{v_{2}}^{(1)}=k_2t_2$ for some $t_2$. 
	Hence,  $f_{v_2}^{(1)}$ is determined by the choice of 
	$f_{v_1}^{(1)}$. Now we will show the existence of $f_{v_3}^{(1)}$.  Let $\{(p_1^{(3,1)}),(p_2^{(3,1)}), \dots,(p_{t_1}^{(3,1)})\}$ and $\{(p_1^{(3,2)}),(p_2^{(3,2)}), \dots,(p_{t_2}^{(3,2)})\}$ be the set of greatest common divisors of the edge labels respectively on $P_{31}$ and $P_{32}$. Now, we apply the equivalence condition of the Chinese Remainder Theorem, which states that the following system:
	\begin{align*}
		f_{v_{3}}^{(1)} &\equiv	f_{v_{1}}^{(1)} \mod [\{(p_{l_1}^{(3,1)}) \mid 1\le l_1 \le t_1\}] \\
		f_{v_{3}}^{(1)} &\equiv	f_{v_{2}}^{(1)} \mod [\{(p_{l_2}^{(3,2)}) \mid 1\le l_2 \le t_2\}] \\
		f_{v_{3}}^{(1)} &\equiv 0  \quad \mod k_3
	\end{align*}
	admits a solution if and only if the following conditions hold:
	\begin{align}
		f_{v_{1}}^{(1)} &\equiv	f_{v_{2}}^{(1)} \mod ([\{(p_{l_1}^{(3,1)}) \mid 1\le l_1 \le t_1\}],[\{(p_{l_2}^{(3,2)}) \mid 1\le l_2 \le t_2\}]) \label{1.2.equality}\\
		f_{v_{1}}^{(1)} &\equiv 0   \mod (k_3,[\{(p_{l_1}^{(3,1)}) \mid 1\le l_1 \le t_1\}]) \label{1.satisfy} \\
		f_{v_{2}}^{(1)} &\equiv 0   \mod (k_3,[\{(p_{l_2}^{(3,2)}) \mid 1\le l_2 \le t_2\}]). \label{2.satisfy}
	\end{align}
	By Lemma \ref{divisible k_1}, the fact that $k_1$ is divisible by $(k_3,[\{(p_{l_1}^{(3,1)}) \mid 1\le l_1 \le t_1\}])$ ensures that Equation \ref{1.satisfy} is satisfied. Similarly,  as	$k_2$ is divisible by $(k_3,[\{(p_{l_2}^{(3,2)}) \mid 1\le l_2 \le t_2\}])$, Equation \ref{2.satisfy} also holds. It only remains to verify that Equation \ref{1.2.equality} holds. We can express \begin{equation*}
		[([\{(p_{l_1}^{(3,1)}) \mid 1\le l_1 \le t_1\}],[\{(p_{l_2}^{(3,2)}) \mid 1\le l_2 \le t_2\}])] 
	\end{equation*} as 
	\begin{gather*}
		[(p_1^{(3,1)},p_1^{(3,2)}),\dots,(p_1^{(3,1)},p_{t_2}^{(3,2)}),(p_2^{(3,1)},p_1^{(3,2)}),\dots,(p_2^{(3,1)},p_{t_2}^{(3,2)}),\dots,(p_{t_1}^{(3,1)},p_{t_2}^{(3,2)})].
	\end{gather*} 
	For spline conditions, we have $	f_{v_{2}}^{(1)} \equiv	f_{v_{1}}^{(1)} \mod [\{(p_l^{(2,1)}) \mid 1\le l \le t\}] $. It follows that 
	since each pair	$(p_i^{(3,1)},p_j^{(3,2)})$ divides some $(p_l^{(2,1)})$, for all $1\le i \le t_1$, $1\le j\le t_2$ and  $1\le l \le t$ Equation \ref{1.2.equality} is satisfied.
	
	By proceeding inductively, we can establish the existence of  $f_{v_{i+1}}^{(1)}$ for all $i=3,\dots ,n-1$. Consequently, a flow-up class  $F^{(1)}=(f_{v_1}^{(1)},f_{v_2}^{(1)},\dots,f_{v_n}^{(1)})$ exists.

\end{proof}

\begin{figure}[h]
	
	\begin{center}
		\begin{tikzpicture}

			\node[main node] (1)  {$m_1 \mathbb{Z}$};
			\node[main node] (2) [right = 2cm of 1]  {$m_3 \mathbb{Z}$};
			\node[main node] (3) [above left = 3cm of 2]  {$m_4 \mathbb{Z}$};
			\node[main node] (4) [right = 2cm of 3]  {$m_2 \mathbb{Z}$};

			\path[draw,thick]
			(1) edge node [below]{$r_3 \mathbb{Z}$} (2)
			(2) edge node [right]{$   r_2 \mathbb{Z} $} (3)
			(1) edge node [left]{$ r_1  \mathbb{Z}  $} (3)
			(3) edge node [above]{$ r_4 \mathbb{Z}$} (4);

		\end{tikzpicture}
	\end{center}
	
	\caption{An edge-labeled graph $(G,\beta)$} \label{flow-up exist}
\end{figure}
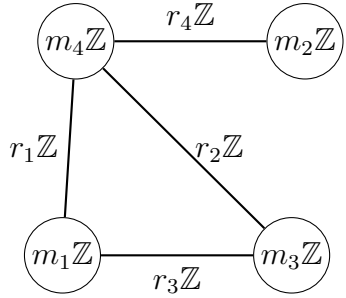

\begin{Ex}
	Consider the edge-labeled graph $(G,\beta)$ as shown in Figure \ref{flow-up exist}.	We let \begin{align*}
		f_{v_1}^{(1)}=k_1= & [m_1,\{(m_j,[ \{(p^{(1,j)})\}]) \mid \forall j\neq 1 \}] \\
		=& [m_1,(m_2,[(p_1^{(1,2)}),(p_2^{(1,2)})]), (m_3,[(p_1^{(1,3)}),(p_2^{(1,3)})]),(m_4,[(p_1^{(1,4)}),(p_2^{(1,4)})])] \\
		=& [m_1,(m_2,[(r_2,r_3,r_4),(r_1,r_4)]),(m_3,[r_3,(r_1,r_2)]),(m_4,[r_1,(r_2,r_3)])].
	\end{align*} 
	We now demonstrate the existence of $f_{v_{2}}^{(1)}$,  and then proceed to construct  $f_{v_{3}}^{(1)}$ and $f_{v_{4}}^{(1)}$.	
	We consider the following system of congruences:
	\begin{align*}
		f_{v_{2}}^{(1)} &\equiv	f_{v_{1}}^{(1)} \mod [(r_2,r_3,r_4),(r_1,r_4)] \\
		f_{v_{2}}^{(1)} &\equiv 0  \quad \mod k_2.
	\end{align*}
	According to Theorem \ref{CRT} (CRT),   this system admits a solution for $f_{v_{2}}^{(1)}$  if and only if \[f_{v_{1}}^{(1)} \equiv 0 \mod ([(r_2,r_3,r_4),(r_1,r_4)],k_2).\]   By Lemma \ref{divisible k_1}, we know that $k_1$ is divisible by  $([(r_2,r_3,r_4),(r_1,r_4)],k_2)$, which ensures the existence of a solution for  $f_{v_{2}}^{(1)}$.
	Next, we establish the existence of  $f_{v_3}^{(1)}$. 
	Applying the Chinese Remainder Theorem, we consider the following system:
	\begin{align*}
		f_{v_{3}}^{(1)} &\equiv	f_{v_{1}}^{(1)} \mod [r_3,(r_1,r_2)] \\
		f_{v_{3}}^{(1)} &\equiv	f_{v_{2}}^{(1)} \mod [(r_2,r_4),(r_1,r_3,r_4)] \\
		f_{v_{3}}^{(1)} &\equiv 0  \quad \mod k_3
	\end{align*}
	admits a solution if and only if the following compatibility conditions are satisfied:
	\begin{align*}
		f_{v_{1}}^{(1)} &\equiv	f_{v_{2}}^{(1)} \mod ([r_3,(r_1,r_2)],[(r_2,r_4),(r_1,r_3,r_4)]) \\
		f_{v_{1}}^{(1)} &\equiv 0   \mod (k_3,[r_3,(r_1,r_2)])  \\
		f_{v_{2}}^{(1)} &\equiv 0   \mod (k_3,[(r_2,r_4),(r_1,r_3,r_4)]) 
	\end{align*}
	where
	$	k_3=[m_3,(m_1,[r_3,(r_1,r_2)]),(m_2,[(r_2,r_4),(r_1,r_3,r_4)]),(m_4,[r_2,(r_1,r_3)])].$
	By Lemma \ref{divisible k_1}, we know that $k_1$ is divisible by $(k_3,[r_3,(r_1,r_2)])$ and $k_2$ is divisible by $(k_3,[(r_2,r_4),(r_1,r_3,r_4)])$. Moreover, it is easy to verify that $[(r_2,r_3,r_4),(r_1,r_4)]$ is divisible by $([r_3,(r_1,r_2)],[(r_2,r_4),(r_1,r_3,r_4)]).$ Thus,  a solution for  $f_{v_{3}}^{(1)}$ exists. It remains to show the existence of $f_{v_{4}}^{(1)}$.
	We consider the following system of congruences:
	\begin{align*}
		f_{v_{4}}^{(1)} &\equiv	f_{v_{1}}^{(1)} \mod [r_1,(r_2,r_3)] \\
		f_{v_{4}}^{(1)} &\equiv	f_{v_{2}}^{(1)} \mod r_4 \\
		f_{v_{4}}^{(1)} &\equiv	f_{v_{3}}^{(1)} \mod [r_2,(r_1,r_3)] \\
		f_{v_{4}}^{(1)} &\equiv 0  \quad \mod k_4.
	\end{align*}
	According to Theorem \ref{CRT} (CRT),   this system admits a solution for $f_{v_{4}}^{(1)}$  if and only if
	\begin{align*}
		f_{v_{1}}^{(1)} &\equiv	f_{v_{2}}^{(1)} \mod (r_4,[r_1,(r_2,r_3)] ) \\
		f_{v_{1}}^{(1)} &\equiv	f_{v_{3}}^{(1)} \mod ([r_1,(r_2,r_3)],[r_2,(r_1,r_3)] ) \\
		f_{v_{2}}^{(1)} &\equiv	f_{v_{3}}^{(1)} \mod (r_4,[r_2,(r_1,r_3)] ) \\
		f_{v_{1}}^{(1)} &\equiv 0 \mod (k_4,[r_1,(r_2,r_3)] ) \\
		f_{v_{2}}^{(1)} &\equiv 0 \mod (k_4,r_4 ) \\
		f_{v_{3}}^{(1)} &\equiv	0 \mod (k_4,[r_2,(r_1,r_3)] ). 
	\end{align*}
	By Lemma \ref{divisible k_1}, we know that $k_1$ is divisible by $(k_4,[r_1,(r_2,r_3)] )$, $k_2$ is divisible by $(k_4,r_4 )$ and $k_3$ is divisible by $(k_4,[r_2,(r_1,r_3)] )$. Moreover, it is easy to verify that $[(r_2,r_3,r_4),(r_1,r_4)]$  is divisible by $(r_4,[r_1,(r_2,r_3)] )$. Additionally,  $ [r_3,(r_1,r_2)]$  is divisible by
	\begin{equation*}
		([r_1,(r_2,r_3)],[r_2,(r_1,r_3)] )=  [(r_1,r_2),(r_1,r_3),(r_2,r_3),(r_1,r_2,r_3)].
	\end{equation*}  and $[(r_2,r_4),(r_1,r_3,r_4)]$ is divisible by $(r_4,[r_2,(r_1,r_3)] )$. Thus, a solution for  $f_{v_{4}}^{(1)}$ exists.

\end{Ex}

\begin{Th} \label{flowup}
	Let $(G,\beta)$ be an edge-labeled graph and  $\mathcal{P}^{i}$  the set of the longest paths $P_{ji}=(m_jr_{i_t}m_{i_t}r_{i_{t-1}}\dots r_{i_1}m_{i_1}r_{i}m_i)$ of $v_i$. Let    $F_k=(f_{kv_i},f_{kv_{i_1}},\dots,f_{kv_{i_t}},f_{kv_j})$  be a minimal spline with $f_{kv_i} \neq 0$ on $(P_{ji},\beta_{ji})$ where $P_{ji} \in \mathcal{P}^{i}$ and $\beta_{ji}=\beta{\mid{P_{ji}}}$.
	Then there exists a flow-up class $F^{(i)}=(0,\dots,0,f_{v_i}^{(i)},\dots,f_{v_n}^{(i)})$ with the first nonzero entry $f_{v_i}^{(i)}$ equal to  the least common multiple of all $f_{kv_i}$.

\end{Th}

\begin{proof}
	In Theorem \ref{exist}, we  established the existence of  $F^{(1)}$. Now, we assert that the entry $f_{v_{j}}^{(i)}$ exists for all $j=i,i+1,\ldots,n$ where $i=2,\ldots,n$. Let  $(G_{zred},\beta_{zred})$ be the result of a graph reduction on all zero-labeled vertices $v_1,\ldots,v_{i-1}$ from $(G,\beta)$ .
	This new graph is not necessarily connected. We choose a connected component subgraph $H$ of $(G_{zred},\beta_{zred})$ containing the vertex $v_i$. By applying Theorem \ref{exist} on $H$, we obtain a flow-up class  $F^{(i)}=(f_{v_i}^{(i)},\dots,f_{v_k}^{(i)})$ on $H$ and extend this flow-up on $(G_{zred},\beta_{zred})$, by setting zeros for $f_{v_t}^{(i)}$ at the  vertices $v_t$ appearing in other connected components. By Proposition \ref{remove},  we conclude that  $F^{(i)}=(0,\dots,0,f_{v_i}^{(i)},\dots,f_{v_n}^{(i)})$ also exists on $(G,\beta)$.

\end{proof}

\begin{Th}\label{basisgeneral}
	Let  $(G,\beta)$  be an edge-labeled graph and  $\mathcal{P}^{i}$  the set of the longest paths $P_{ki}$ of $v_i$. Let    $F_k=(f_{kv_i},f_{kv_{j_1}},\dots,f_{kv_{j_t}})$  be a minimal spline with $f_{kv_i} \neq 0$ on $(P_{ji},\beta_{ji})$ where $P_{ji} \in \mathcal{P}^{i}$ and $\beta_{ji}=\beta{\mid{P_{ji}}}$. Let
	
	\begin{equation*}
		B = \begin{Bmatrix}
			\left( \begin{array}{c}
				f_{v_n}^{(1)}\\
				f_{v_{n-1}}^{(1)}\\
				\vdots \\
				f_{v_3}^{(1)} \\
				f_{v_2}^{(1)}\\
				f_{v_1}^{(1)}
			\end{array} \right) , &

			\left( \begin{array}{c}
				f_{v_n}^{(2)}\\
				f_{v_{n-1}}^{(2)}\\
				\vdots \\
				f_{v_3}^{(2)} \\
				f_{v_2}^{(2)}\\
				0
			\end{array} \right) , &

			\ldots  & ,

			\left( \begin{array}{c}
				f_{v_n}^{(i)}\\
				\vdots\\
				f_{v_i}^{(i)}\\
				0\\
				\vdots \\
				0
			\end{array} \right) , &
			
			\ldots & ,
			
			\left( \begin{array}{c}
				f_{v_n}^{(n)}\\
				0\\
				0\\
				\vdots \\
				0 \\
				0
			\end{array} \right) 
		\end{Bmatrix}
	\end{equation*}
	be a flow-up class set	where each  $F^{(i)} \in \hat{\mathcal{F}_i}$ is minimal  with the first nonzero entry	$f_{v_i}^{(i)}$ equal to  the least common multiple of all $f_{kv_i}$
	for $i=1,2,\ldots,n$. Then $B$ forms a module basis over $\mathbb{Z}$. 
\end{Th}

\begin{proof}
	It follows from Theorem \ref{basis cond}.
\end{proof}

\begin{figure}[h]
	
	\begin{center}
		\begin{tikzpicture}
			
			\node[main node] (1) {$v_6 $};
			\node[main node] (2) [right = 2cm of 1]  {$v_5 $};
			\node[main node] (3) [right = 2cm of 2]  {$v_1 $};
			\node[main node] (4) [right = 2cm of 3]  {$v_9 $};
			\node[main node] (5) [above = 2cm of 3]  {$v_8 $};
			\node[main node] (6) [below = 2cm of 3]  {$v_7 $};
			\node[main node] (7) [above = 2cm of 2]  {$v_4 $};
			\node[main node] (8) [left= 2cm of 7]  {$v_3 $};
			\node[main node] (9) [left = 2cm of 8]  {$v_2 $};

			\path[draw,thick]
			(1) edge node [above]{$ 4 \mathbb{Z}$} (2)
			(2) edge node [above]{$ 3 \mathbb{Z}$} (3)
			(3) edge node [right]{$ 9 \mathbb{Z}$} (5)
			(3) edge node [right]{$  6 \mathbb{Z}$} (6)
			(3) edge node [above]{$ 12 \mathbb{Z}$} (4)
			(2) edge node [right]{$  5 \mathbb{Z}  $} (7)
			(7) edge node [above]{$  8 \mathbb{Z}$} (8)
			(8) edge node [above]{$   10 \mathbb{Z}$} (9);

		\end{tikzpicture}
	\end{center}
	
	\caption{An example of the edge-labeled tree graph $(T_9,\beta)$} \label{extree}
\end{figure}
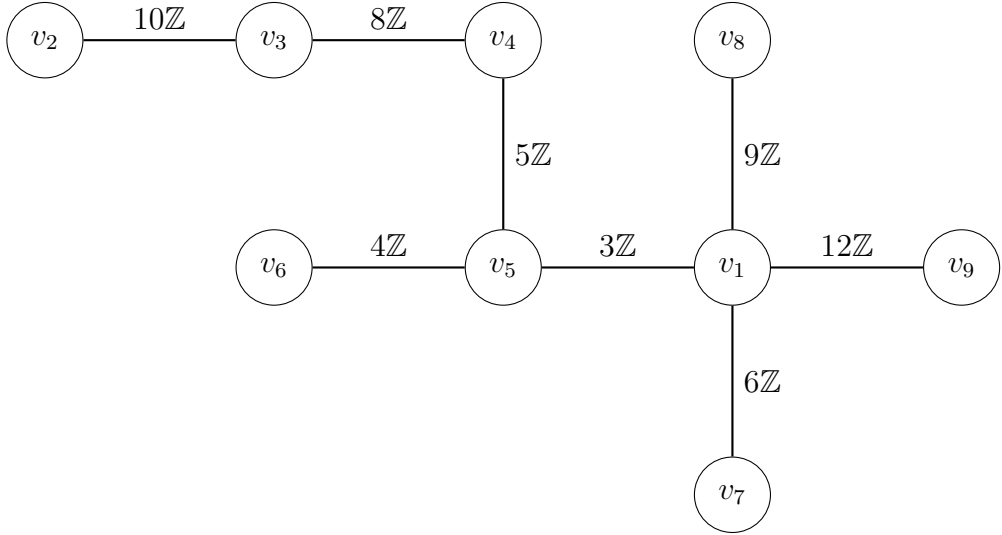	
Now we give an example of trees.
 
\begin{Ex}
	Consider the edge-labeled tree graph $(T_9, \beta)$  with a $\mathbb{Z}$-module $M_{v_i}= m_i \mathbb{Z}$ at each vertex $v_i$ where $$m_1=4, m_2=15, m_3=9, m_4=8, m_5 = 12, m_6=3, m_7=2, m_8=5, m_9=6.$$
Each edge $uv$ is labeled by  $ l_{uv}=\beta(uv)$, where 	$\beta(uv)= \mathbb{Z} / r_{uv} \mathbb{Z} \subset \mathbb{Z}$, and the corresponding submodule $r_{uv} \mathbb{Z}$ is associated to the edge, as illustrated in Figure~\ref{extree}.
	Define $F^{(i)}= (0,\ldots,0,f_{v_i}^{(i)},\ldots,f_{v_9}^{(i)})$ 
	to be a flow-up class  with $f_{v_i}^{(i)} \neq 0$ for each $i$ with $1\le i \le 9$ and $f_{v_s}^{(i)} =0$ for all $s <i$. 
	By Theorem \ref{longmult} and \ref{minimal}, we obtain  the first nonzero smallest entry ${f_{v_i}^{(i)}}$ of $F^{(i)}$ as follows:
	\begin{align*}
	f_{v_1}^{(1)}= &[4,(5,9),(2,6),(6,12),(3,12),(3,4,3),(8,3,5),(9,8,5,3),(15,10,8,5,3)]=12\\ f_{v_2}^{(2)} = & [15,(9,10),(8,8,10),([12,3],5,8,10),(3,4,5,8,10)]=30 \\
 f_{v_3}^{(3)}= &	[9,10,(8,8),([12,3],5,8),(3,4,5,8)]=360\\
  f_{v_4}^{(4)}= &[8,8,([12,3],5),(3,4,5)]=8 \\
  f_{v_5}^{(5)}= & [12,3,5,(3,4)]=60 \\
  f_{v_6}^{(6)}=& [3,4]=12 \\
  f_{v_7}^{(7)}= &[2,6]=6 \\
  f_{v_8}^{(8)}= & [5,9]=45 \\
  f_{v_9}^{(9)} = & [6,12]=12.
	\end{align*}	
Using these and solving spline conditons we obtain a set $B$ of minimal flow-up classes:
	\small	\begin{equation*}
		B = \begin{Bmatrix}
			\left( \begin{array}{c}
				0\\
				30\\
				0\\
				0\\
				0\\
				0\\
				0\\
				0\\
				12
			\end{array} \right) , &

			\left( \begin{array}{c}
				0\\
				0\\
				0\\
				0\\
				0\\
				0\\
				0\\
				30\\
				0
			\end{array} \right) , &
			
			\left( \begin{array}{c}
				0\\
				0\\
				0\\
				0\\
				0\\
				0\\
				360\\
				0\\
				0
			\end{array} \right) , &
			
			\left( \begin{array}{c}
				0\\
				0\\
				0\\
				0\\
				48\\
				8\\
				0\\
				0\\
				0
			\end{array} \right) , &

			\left( \begin{array}{c}
				0\\
				0\\
				0\\
				0\\
				60\\
				0\\
				0\\
				0\\
				0
			\end{array} \right) , &

			\left( \begin{array}{c}
				0\\
				0\\
				0\\
				12\\
				0\\
				0\\
				0\\
				0\\
				0
			\end{array} \right) , &

			\left( \begin{array}{c}
				0\\
				0\\
				6\\
				0\\
				0\\
				0\\
				0\\
				0\\
				0
			\end{array} \right) , &

			\left( \begin{array}{c}
				0\\
				45\\
				0\\
				0\\
				0\\
				0\\
				0\\
				0\\
				0
			\end{array} \right) , &
			
			\left( \begin{array}{c}
				12\\
				0\\
				0\\
				0\\
				0\\
				0\\
				0\\
				0\\
				0
			\end{array} \right)  &
			
		\end{Bmatrix}.
	\end{equation*}
	\normalsize
	By Theorem \ref{basisgeneral}, it forms a module basis over $\mathbb{Z}$.
\end{Ex}

\begin{figure}[h!]	
	\begin{center}
		\begin{tikzpicture}
			\node[main node] (1) {$m_1 \mathbb{Z}$};
			\node[main node] (2) [right = 2cm of 1]  {$m_2 \mathbb{Z}$};
			\node[main node] (3) [right = 2cm of 2]  {$m_3 \mathbb{Z}$};
			\node[right=.1cm of 3] {$\dots$};
			
			\node[main node] (4) [right = 1cm of 3]  {$m_{n-2} \mathbb{Z}$};
			\node[main node] (5) [right = 2cm of 4]  {$m_{n-1} \mathbb{Z}$};
			\node[main node] (6) [right = 2cm of 5]  {$m_n \mathbb{Z}$};

			\path[draw,thick]
			(1) edge node [above]{$ r_1 \mathbb{Z}$} (2)
			(2) edge node [above]{$ r_2 \mathbb{Z}$} (3)

			(4) edge node [above]{$ r_{n-2} \mathbb{Z}$} (5)
			(5) edge node [above]{$ r_{n-1} \mathbb{Z}$} (6);

			\draw [out=35,in=30,looseness=0.5] (1) to node[ above]{$ r_n \mathbb{Z}$}  (6);

		\end{tikzpicture}
	\end{center}
	
	\caption{An edge-labeled cycle graph $(C_n,\beta)$} \label{cycle n}
\end{figure}
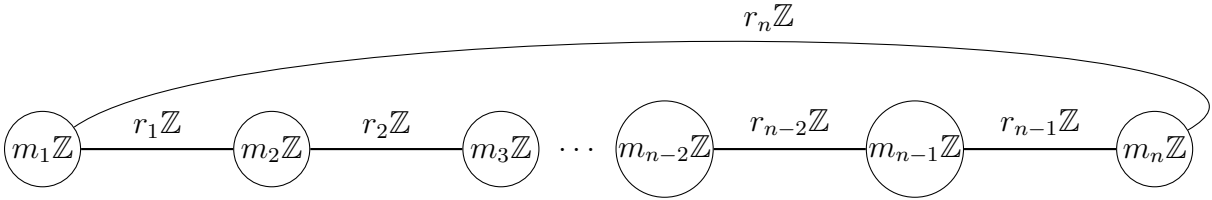

The next example is on arbitrary ordered cycles.
\begin{Ex}
Consider the edge-labeled cycle graph $(C_n,\beta)$  as shown in Figure \ref{cycle n}. Let $F^{(i)}=(0,\dots,0,f_{v_i}^{(i)},\dots,f_{v_n}^{(i)})$ be a  
 flow-up class on $(C_n,\beta)$, where each $f_{v_i}^{(i)}\neq 0$ and $f_{v_s}^{(i)}=0$ for all $s<i$. The value of $f_{v_i}^{(i)}$ can be determined using the longest paths. Specifically, the set of longest paths to $v_1$ is given by	$\mathcal{P}^{1}= \{P_{n1},P_{21}\}$,   where $P_{n1}=(v_ne_{n-1}v_{n-1}\cdots e_2v_{2}e_{1}v_1)$ and $P_{21}=(v_2e_2v_3\cdots e_{n-1}v_ne_nv_1)$. Let $f_{11}$ and  $f_{12}$  denote the first nonzero entries of a minimal flow-up class along the paths $(P_{n1},\beta_{n1})$ and $(P_{21},\beta_{21})$, respectively.  Then, by Theorem  \ref{longmult}, $f_{v_1}^{(1)}$ is a multiple of $[f_{11},f_{12}]$ where
 \begin{align*}
 	f_{11} = & [m_1, (m_2,r_1),(m_3,r_1,r_2),\dots,(m_n,r_1,\dots,r_{n-1})], \\
 	f_{12}= & [m_1, (m_n,r_n),(m_{n-1},r_{n-1},r_n),\dots,(m_2,r_2,\dots,r_n)].
 \end{align*} By Theorem \ref{exist}, there exists a minimal flow-up $F^{(1)}$ with the smallest element $f_{v_1}^{(1)}=[f_{11},f_{12}]$. Similarly,
 for any $2\le i \le n$, we aim to derive a similar condition for $f_{v_i}^{(i)}$.  Given that 
  $f_{v_s}^{(i)} = 0$ for all $s < i$, we begin by removing the vertices with zero labels and relabel labels of the vertices to connecting zero labels to obtain a reduced graph. The reduced graph $(G_{zred},\beta_{zred})$ corresponds to the path graph illustrated in Figure \ref{reducedpath}, where $m_n'=[m_n,r_n]$ and $m_i'=[m_i,r_{i-1}]$. Applying Theorem \ref{pexists} of paths, we obtain the following expression for $f_{v_i}^{(i)}$:
 \begin{align*}
 	f_{v_i}^{(i)} & =  [m_i^{'}, (m_{i+1},r_i),(m_{i+2},r_i,r_{i+1}),\dots,(m^{'}_n,r_i,r_{i+1},\dots,r_{n-1})] \\
 	&=  [m_i,r_{i-1}, (m_{i+1},r_i),(m_{i+2},r_i,r_{i+1}),\dots,([m_n,r_n],r_i,r_{i+1},\dots,r_{n-1})] 
 \end{align*}
 for all $i=2,\dots,n$.
  Therefore, by Theorem \ref{basisgeneral}, the set $B=\{F^{(1)},F^{(2)},\dots,F^{(i)},\dots,F^{(n)}\}$ where each $F^{(i)} \in \hat{\mathcal{F}_i}$ is a minimal flow-up class with the first nonzero entry given by:
  \begin{align*}
  	f_{v_1}^{(1)}= & [m_1,(m_2,r_1),(m_3,r_1,r_2),\dots,(m_n,r_1,\dots,r_{n-1}),(m_n,r_n),(m_{n-1},r_{n-1},r_n),\dots,(m_2,r_2,\dots,r_n)], \\
  	f_{v_i}^{(i)}= & [m_i,r_{i-1}, (m_{i+1},r_i),(m_{i+2},r_i,r_{i+1}),\dots,(m_n,r_i,r_{i+1},\dots,r_{n-1}),(r_i,r_{i+1},\dots,r_{n-1},r_n)]
  \end{align*}
  for all $i=2,\dots,n$, forms a $\mathbb{Z}$-module basis.

  \begin{figure}[h!]	
 	\begin{center}
 		\begin{tikzpicture}
 			\node[main node] (1) {$ m_i^{'}\mathbb{Z}$};
 			\node[main node] (2) [right = 2cm of 1]  {$m_{i+1} \mathbb{Z}$};
 			\node[main node] (3) [right = 2cm of 2]  {$m_{i+2} \mathbb{Z}$};
 			\node[right=.1cm of 3] {$\dots$};
 			
 			\node[main node] (4) [right = 1cm of 3]  {$m_{n-2} \mathbb{Z}$};
 			\node[main node] (5) [right = 2cm of 4]  {$m_{n-1} \mathbb{Z}$};
 			\node[main node] (6) [right = 2cm of 5]  {$m_n^{'} \mathbb{Z}$};

 			\path[draw,thick]
 			(1) edge node [above]{$r_i \mathbb{Z}$} (2)
 			(2) edge node [above]{$ r_{i+1} \mathbb{Z}$} (3)

 			(4) edge node [above]{$ r_{n-2} \mathbb{Z}$} (5)
 			(5) edge node [above]{$ r_{n-1} \mathbb{Z}$} (6);

 		\end{tikzpicture}
 	\end{center}
 	
 	\caption{Reduced graph $(G_{zred},\beta_{zred})$} \label{reducedpath}
 \end{figure}
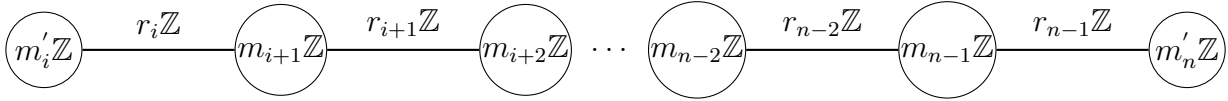

\end{Ex}

\section*{Acknowledgement}

This article is based on doctoral research conducted at the Hacettepe University Graduate School of Science and Engineering.


\begin{thebibliography}{00}
	
		\bibitem{AD}  S. Altınok, G. Dilaver, \textit{Minimum generating sets for complete graphs} (2023). 	\href{
		https://acikerisim.isikun.edu.tr/xmlui/handle/11729/5619
	}{TWMS Journal of Applied and Engineering Mathematics 13.3}.
	
	
	\bibitem{AS2019}  S. Altınok, S. Sarıoğlan, \textit{Flow-up bases for generalized spline modules on arbitrary graphs}, J. Algebra Appl. (2019). \href{https://www.worldscientific.com/doi/10.1142/S0219498821501802}{doi: 10.1142/S0219498821501802}.
	
	\bibitem {AS2021} S. Altınok, S. Sarıoğlan. \textit{Basis criteria for generalized spline modules via determinant} (2021).  \href{https://www.sciencedirect.com/science/article/pii/S0012365X2030409X?casa_token=cYY9HB-L-RQAAAAA:gHAF9y3wOgPq5bPaW85veL_GNIzyy6j5sxzdL_aDAE_7yISs4q6A8VG0p-fdA3ORPdz4iH2jLFfk}{Discrete Mathematics, Vol.344, No. 2,}.
	

	\bibitem{BHKR} N. Bowden, S. Hagen, M. King, S. Reinders, \textit{Bases and structure constants of generalized splines with integer coefficients on cycles} (2015). \href{https://arxiv.org/abs/1502.00176}{arXiv:1502.00176v1}.
	
	\bibitem{Gjoni} E. Gjoni, \textit{Basis criteria for n-cycle integer splines}, Senior Projects Spring (2015).
	
	\bibitem{GPT} S. Gilbert, S. Polster, J. Tymoczko, \textit{Generalized splines on arbitrary graphs}, Pacific J. Math. (2016) 281 (2) 333-364.
	
	\bibitem{GKM}  M. Goresky, R. Kottwitz, R. MacPherson, \textit{Equivariant cohomology, Koszul duality, and the localizations theorem}, Invent. Math. 131 (1) (1998) 25-83.
	
	\bibitem{HMR}  M. Handschy, J. Melnick, S. Reinders, \textit{Integer generalized splines on cycles} (2014). arXiv:1409.1481, \href{https://arxiv.org/abs/1409.1481}{arXiv:1409.1481}.
	
	
	\bibitem{Mahdavi} E. R. Mahdavi, \textit{Integer generalized splines on the diamond graph}.  Bard College Senior Projects Spring (2016). 
	
	
	
	\bibitem{Tymoczko} J. Tymoczko, \textit{Splines in geometry and topology} (2015). \href{
		https://doi.org/10.48550/arXiv.1511.07555
	}{arXiv.1511.07555}.
	
	
	
	
\end{thebibliography}
\end{document}